\documentclass[12pt, reqno]{amsart}
\usepackage[T1]{fontenc}
\usepackage[latin1]{inputenc}
\usepackage{typearea}
\usepackage{lipsum}
\usepackage{layout}
\usepackage[hmarginratio=1:1,margin=3cm]{geometry}
\usepackage[hidelinks]{hyperref}
\usepackage{ulem}
\usepackage{amsmath}
\usepackage{amsfonts}
\usepackage{amssymb}
\usepackage{latexsym}
\usepackage{enumerate}
\usepackage{verbatim}
\usepackage{amsthm}
\usepackage[all]{xy}
\usepackage{hhline}
\usepackage{epsf} 
\usepackage{cite}
\usepackage{mathrsfs}
\numberwithin{equation}{section}

\usepackage{color}

\newtheorem{theorem}{{\sc Theorem}}[section]
\newtheorem{cor}[theorem]{{\sc Corollary}}

\newtheorem{lemma}[theorem]{{\sc Lemma}}
\newtheorem{prop}[theorem]{{\sc Proposition}}

\theoremstyle{remark}
\newtheorem{remark}[theorem]{{\sc Remark}}

\theoremstyle{definition}

\newcommand{\R}{\mathbb{R} }

\newcommand{\N}{\mathbb{N} }

\newcommand{\F}{\mathcal{F}}

\newcommand{\W}{\mathcal{W}}
\newcommand{\K}{\mathcal{K}}

\newcommand{\Tr}{\textnormal{Tr}}
\newcommand{\Prob}{\mathbb{P}}
\newcommand{\E}{\mathbb{E}}

\renewcommand{\P}{\mathbb{P}}

\providecommand{\abs}[1]{\lvert #1\rvert}
\providecommand{\babs}[1]{\bigl\lvert #1\bigr\rvert}

\providecommand{\Opnorm}[1]{\lVert #1\rVert_{\op}}
\providecommand{\Enorm}[1]{\lVert #1\rVert_2}
\providecommand{\HSnorm}[1]{\lVert #1\rVert_{\HS}}

\DeclareMathOperator{\Var}{Var}

\DeclareMathOperator{\dom}{dom}

\DeclareMathOperator{\Lip}{Lip}
\DeclareMathOperator{\op}{op}
\DeclareMathOperator{\Hess}{Hess}
\DeclareMathOperator{\HS}{H.S.}

\renewcommand{\phi}{\varphi}
\renewcommand{\epsilon}{\varepsilon}

\renewcommand{\rho}{\varrho}

\usepackage{etoolbox}
\patchcmd{\section}{\scshape}{\bfseries}{}{}
\makeatletter
\renewcommand{\@secnumfont}{\bfseries}
\makeatother

\begin{document}

\title[Fourth moment theorems]{Fourth moment theorems on the Poisson space in any dimension}

\author{Christian D\"obler, Anna Vidotto, Guangqu Zheng}
\thanks{\noindent Universit\'{e} du Luxembourg, Unit\'{e} de Recherche en Math\'{e}matiques \\
E-mails: christian.doebler@uni.lu, anna.vidotto@uni.lu, guangqu.zheng@uni.lu
}

\begin{abstract}  
We extend to any dimension the quantitative fourth moment theorem on the Poisson setting, recently proved by C. D\"obler and G. Peccati (2017). In particular, by adapting the exchangeable pairs couplings construction introduced by I. Nourdin and G. Zheng (2017) to the Poisson framework, we prove our results under the weakest possible assumption of finite fourth moments. 
This yields a Peccati-Tudor type theorem, as well as an optimal improvement in the univariate case.

Finally, a transfer principle ``from-Poisson-to-Gaussian'' is derived, which is closely related to the universality phenomenon for homogeneous multilinear sums. 

\end{abstract}

\maketitle

\section{Introduction and main results}\label{intro}

\bigskip

\subsection{Outline} In the recent paper \cite{DP17}, the authors succeeded in proving exact quantitative \textit{fourth moment theorems} for multiple Wiener-It\^{o} integrals on the Poisson space. 
Briefly, their method consisted in extending the spectral framework initiated by the remarkable paper \cite{Led12}, and further refined by \cite{ACP}, from the situation of a diffusive Markov generator to the non-diffusive Ornstein-Uhlenbeck generator on the Poisson space. 
The principal aim of the present article is to extend the results from \cite{DP17} to the multivariate case of vectors of multiple integrals.
In view of the result of Peccati and Tudor \cite{PeTu} on vectors of multiple integrals on a Gaussian space, we are in particular interested in discussing the relationship between coordinatewise convergence and joint convergence to normality. 
Indeed, one of our main achievements is a complete quantitative version of a Peccati-Tudor type theorem on the Poisson space (see Theorem \ref{mtmd} and Corollary \ref{cormd}).

Furthermore, still keeping the spectral point of view as in \cite{DP17}, by replacing the rather intrinsic techniques used there with an adaption of a recent construction of \textit{exchangeable pairs couplings} from \cite{NZ17}, we can even remove certain technical conditions which seem inevitable in order to justify the computations in \cite{DP17}. In this way, we are able to prove our results under the \textit{weakest possible assumption} of finite fourth moments.
In the univariate case, our strategy  provides an optimal improvement of the Wasserstein bound given in Theorem 1.3 of \cite{DP17} and, a fortiori, of the associated qualitative fourth moment theorem on the Poisson space 
(see Corollary 1.4 in \cite{DP17}).

 Whereas in the diffusive situation of \cite{Led12} and \cite{ACP}, the authors were able to identify very general quite weak conditions that ensure the validity of fourth moment theorems, we stress that this is not possible in the general 
non-diffusive situation. This is because, due to the fact that one only has an approximate chain rule, an additional remainder term in the bound necessarily arises. On the Poisson space, fortunately, it turns out that also this remainder term can be completely controlled in terms of the fourth cumulant (via inequality \eqref{DP3.2} below). However, there are examples of non-diffusive situations where no fourth moment theorem holds. Indeed, in the paper \cite{DK17} the authors give 
counterexamples for the setup of infinite Rademacher sequences that show that fourth moment theorems do not always hold in this case but that one also has to take into account the maximal influence of every single member of the random sequence. 
Thus, we emphasize that we are here dealing with a very peculiar example of a non-diffusive Markov generator where, quite coincidentally, all relevant terms can be reduced to just the fourth cumulant. 


\subsection{Motivation and related works} 
The so-called \textit{fourth moment theorem} by Nualart and Peccati \cite{NP05} states that a normalized sequence of multiple Wiener-It\^{o} integrals of fixed order on a Gaussian space converges in distribution to a standard normal random variable $N$, if and only if the corresponding sequence of fourth moments converges to $3$, i.e.\ to the fourth moment of $N$. For  future  reference, we give a precise statement of this result: 

\begin{theorem}[\cite{NP05}] \label{FMT-NP} Let $F_n = I_q^W(f_n)$ be a sequence of multiple Wiener-It\^o integral of order  $q\geq 2$, associated with a Brownian motion {\rm $(W_t, t\in\R_+)$} such that 
 $f_n\in L^2(\R_+^q)$ is symmetric for each $n\in\N$, and 
          $
   \lim_{n\to+\infty} \E[ F_n^2] = \sigma^2 > 0
            $. 
Then, the following statements are equivalent:
\begin{enumerate}
 \item[\rm (1)] $\E[ F_n^4]\to 3\sigma^4$, as $n\to+\infty$.

 \item[\rm (2)] $F_n$ converges in law to a  Gaussian distribution $\mathcal{N}(0, \sigma^2)$, as $n\to+\infty$.

 \item[\rm (3)] For each $r\in\{1, 2, \ldots, q-1\}$, $ \| f_n\otimes_r f_n  \| _{L^2(\R_+^{2q-2r})}\to 0$, as $n\to+\infty$.
\end{enumerate}
The contraction $ f_n\otimes_r f_n$ is defined as in Section \ref{integrals}. See {\rm \cite{NouPecbook}} for any unexplained notions and  notation of Gaussian analysis.
\end{theorem}

Note that such a result significantly simplifies the method of moments for sequences of random variables inside a fixed Wiener chaos. 
In the years after the appearance of \cite{NP05}, this result has been extended and refined in many respects.  While \cite{PeTu} provided a significant multivariate extension (see Theorem \ref{PTudor-G}), the paper \cite{NP-ptrf} combined Stein's method of normal approximation and Malliavin calculus in order to yield quantitative bounds for the normal approximation of general smooth functionals on the Wiener space. We refer to the monograph \cite{NouPecbook} for a comprehensive treatment of the so-called 
\textit{Malliavin-Stein approach} on the Wiener space and  of results obtained in this way. One remarkable result quoted from \cite{NouPecbook} is that, if $F$ is a normalised multiple Wiener-It\^{o} integral of order $q\geq1$ on a Gaussian space, then one has the bound 
\begin{equation}\label{npbound}
 d_{\rm TV}(F,N)\leq2\sqrt{\frac{q-1}{3q}\bigl(\E[F^4]-3\bigr)}\,,
\end{equation}
where $d_{\rm TV}$ denotes the total variation distance between the laws of two real random variables. The techniques developed in \cite{NP-ptrf} have also been adapted to non-Gaussian spaces which admit a Malliavin calculus structure: for instance, 
the papers \cite{PSTU, PZ1, Sch16, ET14} deal with the Poisson space case, whereas \cite{NPR-ejp, KRT1, KRT2,Zheng15} develop the corresponding techniques for sequences of independent Rademacher random variables. 
The question about general fourth moment theorems on these spaces, however, has remained open in general, until the two recent articles \cite{DP17} and \cite{DK17}.


\subsection{General framework}
Let us fix a measurable space $(\mathcal{Z},\mathscr{Z})$, endowed with a $\sigma$-finite measure $\mu$. 
        We let
\[
       \mathscr{Z}_\mu:=\{B\in\mathscr{Z}\,:\,\mu(B)<\infty\}
\]
and define
\begin{equation*}
\eta=\{\eta(B)\,:\,B\in\mathscr{Z}\}
\end{equation*}
to be a \textit{Poisson random measure} on $(\mathcal{Z},\mathscr{Z})$ with \textit{control} $\mu$, defined on a suitable probability space $(\Omega,\F,\Prob)$. By definition, the distribution of $\eta$ is completely determined by the following two properties: 
\begin{itemize}
\item[(i)] for each finite sequence $B_1,\dotsc,B_m\in\mathscr{Z}$ of pairwise disjoint sets, the random variables $\eta(B_1),\dotsc,\eta(B_m)$ are independent;

\item[(ii)] for every $B\in\mathscr{Z}$, the random variable $\eta(B)$ has the Poisson distribution with mean $\mu(B)$.
\end{itemize}
Here, we have extended the family of Poisson distributions to the parameter region $[0,+\infty]$ in the usual way. 
For $B\in\mathscr{Z}_\mu$, we also define 
$\widehat{\eta}(B):=\eta(B)-\mu(B)$ and denote by 
$$
\widehat{\eta}=\{\widehat{\eta}(B)\,:\,B\in\mathscr{Z}_\mu\}
$$
the \textit{compensated Poisson random measure} associated with $\eta$. 
Before stating our main results, we need to define some objects from stochastic analysis on the Poisson space. For a detailed discussion see, among others, \cite{Lastsv} and \cite{LPbook}.

 \bigskip

For $q\in\N_0: =\{0,1,2,\dotsc\}$ and  $f\in L^2(\mu^q)$, we denote by $I^\eta_q(f)$ the $q$-th order \textit{multiple Wiener-It\^o integral} of $f$ with respect to $\widehat{\eta}$. Let $L$ be the generator of the \textit{Ornstein-Uhlenbeck semigroup} with respect to $\eta$, then it is well known that $-L$ is diagonalizable  on $L^2(\Prob)$ with discrete spectrum $\N_0$
and that, for 
$q\in\N_0$, $F$ is an eigenfunction of $-L$ with eigenvalue $q$, if and only if $F=I^\eta_q(f)$ for some $f\in L^2(\mu^q)$. The corresponding eigenspace $C_q$ will be called the \textit{$q$-th Poisson Wiener chaos} associated with $\eta$. In particular, $C_0 = \R$.

\subsection{Main results in the one-dimensional case}
Recall that the \textit{Wasserstein distance} between (the distributions of) two real random variables $X$ and $Y$ in $L^1(\P)$ is defined by 
\begin{equation*}
 d_\W(X,Y):=\sup_{h\in\Lip(1)}\babs{\E[h(X)]-\E[h(Y)]}\,,
\end{equation*}
where $\Lip(1)$ denotes the class of all $1$-Lipschitz functions on $\R$.\\

In the univariate case, our main result reads as follows.

\begin{theorem}[Fourth moment bound on the Poisson space]\label{mt1d}
 Fix an integer $q\geq1$ and let $F\in C_q$ be  such that $\sigma^2:=\E[F^2]>0$ and $\E[F^4]<\infty$. Then, with $N$ denoting a standard normal random variable, we have 
 the bounds:
 \begin{align}
  d_\W(F,\sigma N)&\leq \left(  \frac{2q-1}{\sigma q \sqrt{2\pi}} + \frac{2}{3\sigma} \sqrt{\frac{4q-3}{q}}  \right)  \sqrt{\E[ F^4] - 3\sigma^4} \label{1db1} \\
  &  \leq  \left(  \frac{1}{\sigma}   \sqrt{\frac{2}{\pi}} + \frac{4}{3\sigma}  \right)  \sqrt{\E[ F^4] - 3\sigma^4}\,.\label{1db2}  
 \end{align}
\end{theorem}

Theorem \ref{mt1d} immediately implies the following qualitative statement, which is analogous to the Nualart-Peccati theorem \cite{NP05} on a Gaussian space.

\begin{cor}[Fourth moment theorem on the Poisson space]\label{cor1d}
For each $n\in\N$, let $q_n\in\N$ and  $F_n\in C_{q_n}$ satisfy
\[
\lim_{n\to\infty}\E\bigl[F_n^2\bigr]  =1      \quad    \text{and}     \quad        \lim_{n\to\infty}\E\bigl[F_n^4\bigr]=3\,.
\]
Then, the sequence $(F_n)_{n\in\N}$ converges in distribution to a standard normal random variable $N$.
 \end{cor}

 \begin{remark}\label{rem1d} 
 \begin{enumerate}[(a)]
 \item Theorem \ref{mt1d} and Corollary \ref{cor1d} are genuine improvements of Theorem 1.3 and Corollary 1.4 from \cite{DP17}, respectively, since they do not require any additional regularity from the involved multiple integrals 
   like e.g. \textbf{Assumption A} in \cite{DP17}. The main reason for the appearance of such a condition in \cite{DP17} was that certain intrinsic tools used there, notably the \textit{Mecke formula} and a pathwise representation of the 
   Ornstein-Uhlenbeck generator $L$, require $L^1(\P\otimes\mu)$-integrability conditions. It is the reconciliation of such conditions with the $L^2$ nature of the objects under consideration which necessitated these assumptions. 
   As will become clear from our proofs in Section \ref{proofmd}, such conditions can be completely avoided using an adaptation of exchangeable pairs couplings introduced in \cite{NZ17}.
 \item  In view of the well-known relation $d_\K(F,N)\leq  \sqrt{ d_\W(F,N) }$ between  the Kolmogorov distance and the Wasserstein distance, 
 one can obtain the fourth moment bound in the Kolmogorov distance with order $1/4$ from Theorem \ref{mt1d}, under the weakest possible assumption of finite fourth moment. However, the techniques applied in the present paper do not seem capable of proving a bound of order $1/2$ in the Kolmogorov distance $d_\K(F,N)$.
 \item Indeed, it is an open question whether there is a general bound in the Kolmogorov distance via Stein's method of  exchangeable pairs leading to the same accuracy as the 
   one in the Wasserstein distance. It is worth noting that the authors of \cite{DP17} were able to obtain, under a certain {\it local version} of \textbf{Assumption A} therein, the fourth moment bound in the Kolmogorov distance: 
   \[   
              d_\K(F, N) \leq C\, \sqrt{\E[F^4] - 3\sigma^4}\,, 
 \]
    where $F\in C_q$ with $q\in\N$ and $C$ is a numerical constant. See \cite{DP17} for more details.
   \end{enumerate}
  \end{remark}

In the particular case where
 \begin{center}
\qquad\quad   $\eta$ is a Poisson random measure on $\R_+$ with Lebesgue intensity,   \qquad   (\#)  \end{center}
 we observe the following  {\it transfer principle} that is of independent interest.
 
\quad

\begin{prop}\label{trans-p} Assume {\rm (\#)} and $(W_t,t\in\R_+)$ is a standard Brownian motion.  Given $p\in\N$, $f_n\in L^2(\R_+^p)$ symmetric for each $n\in\N$ such that 
$$  \lim_{n\to+\infty} p! \, \| f_n \| _{L^2(\R_+^p)}^2  = 1 \,\, , $$
then the following implications holds {\rm ($N\sim\mathcal{N}(0,1)$)}
$$ \lim_{n\to+\infty} \E\big[ I_p^\eta(f_n)^4 \big] = 3 \Longrightarrow    \lim_{n\to+\infty} \E\big[ I_p^W(f_n)^4 \big] = 3 \Longrightarrow  \lim_{n\to+\infty} d_{\rm TV}\big( I_{p}^W(f_n), N \big) = 0 \,\,.  $$
\end{prop}

\begin{remark}
This  transfer principle ``from-Poisson-to-Gaussian''  is closely related to the universality of Gaussian Wiener chaos and Poisson Wiener chaos,  see Section \ref{univers}.  It is also worth pointing out that the transfer principle ``from-Gaussian-to-Poisson'' does {\it  not} hold  true, due to a counterexample given in    \cite{BPjfa}, See Proposition 5.4 therein.

\end{remark}  
 
\subsection{Main results in the multivariate case}

In this subsection, let us fix  integers $d\geq 2$ and $1\leq q_1\leq q_2\leq\ldots\leq q_d$ and consider a random vector 
\begin{equation*}
 F:=(F_1,\dotsc,F_d)^T\,,
\end{equation*}
where  $F_j\in C_{q_j}$, $1\leq  j \leq  d$. We will further assume that $F_j\in L^4(\P)$ for each $j\in\{1,\dotsc,d\}$.
Furthermore, we denote by $\Sigma:=(\Sigma_{i,j})_{i,j=1,\dotsc,d}$ the covariance matrix of $F$, i.e. $\Sigma_{i,j}=\E[F_iF_j]$ for $1\leq i,j\leq d$. Note that $\Sigma_{i,j}=0$ whenever $q_i\not=q_j$ due to the orthogonality properties of 
multiple integrals (see Section 2.1), and hence  $\Sigma$ is always a block diagonal matrix. Denote by $N=(N_1,\dotsc,N_d)^T$ a centred Gaussian random vector with the same covariance matrix $\Sigma$.

In order to formulate our bounds, we need to fix some further notation: for a vector $x=(x_1,\dotsc,x_d)^T\in\R^d$,   we denote by $\Enorm{x}$ its \textit{Euclidean norm}
 and for a matrix $A\in\R^{d\times d}$, we denote by $\Opnorm{A}$ 
the \textit{operator norm} induced by the Euclidean norm, i.e., 
\[\Opnorm{A}:= \sup\{ \Enorm{Ax}\,:\ \Enorm{x} = 1\}\,.\]
More generally, for a $k$-multilinear form $\psi:(\R^d)^k\rightarrow\R$, $k\in\N$, we define the \textit{operator norm}
\[
       \Opnorm{\psi}:=\sup\left\{\abs{\psi(u_1,\ldots,u_k)}\,:\, u_j\in\R^d,\, \Enorm{u_j}=1,\, j=1,\ldots,k\,\right\}\, .
\]
Recall that for a function $h:\R^d\rightarrow\R$, its minimum Lipschitz constant $M_1(h)$ is given by
\[
          M_1(h):=\sup_{x\not=y}\frac{\abs{h(x)-h(y)}}{\Enorm{x-y}}\in[0,\infty] \, .
\]
If $h$ is differentiable, then $M_1(h)=\sup_{x\in\R^d}\Opnorm{Dh(x)}$. 
More generally, for $k\geq1$ and a $(k-1)$-times differentiable function $h:\R^d\rightarrow\R$, we set
\[M_k(h):=\sup_{x\not=y}\frac{\Opnorm{D^{k-1}h(x)-D^{k-1}h(y)}}{\Enorm{x-y}}\,,\]
viewing the $(k-1)$-th derivative $D^{k-1}h$ of $h$ at any point $x$ as a $(k-1)$-multilinear form.
Then, if $h$ is   $k$-times differentiable, we have 
$$
M_k(h)=\sup_{x\in\R^d}\Opnorm{D^kh(x)} \, .
$$

Recall that, for two matrices $A,B\in\R^{d\times d}$, their \textit{Hilbert-Schmidt inner product} is defined by 
\begin{eqnarray*}
\langle A,B\rangle_{\HS}:=\Tr\bigl(AB^T\bigr)=\Tr\bigl(BA^T\bigr)=\Tr\bigl(B^TA\bigr)=\sum_{i,j=1}^d A_{i,j}B_{i,j}\,.
\end{eqnarray*}
Thus, $\langle\cdot,\cdot\rangle_{\HS}$ is just the standard inner product on $\R^{d\times d}\cong \R^{d^2}$.
The corresponding \textit{Hilbert-Schmidt norm} will be denoted by $\HSnorm{\cdot}$.  
With this notion at hand, following \cite{ChaMe08} and \cite{Meck09}, for $k=2$ we finally define 
\[\widetilde{M}_2(h):=\sup_{x\in\R^d}\HSnorm{\Hess h(x)}\,,\]
where $\Hess h$ is the \textit{Hessian matrix} corresponding to $h$.
Note that for a   symmetric matrix $A\in\R^{d\times d}$ with eigenvalues $\lambda_1(A)\leq \ldots \leq \lambda_d(A)$, one has
\[  \| A \| _\text{H.S.}^2 = \sum_{j=1}^d \lambda_j(A)^2 \leq d \, \max \{ \vert \lambda_1(A)\vert^2,  \ldots,   \vert \lambda_d(A) \vert^2 \}  = d \| A \| ^2_\text{op} \,\, .\]
From this, it  follows immediately that $\widetilde{M}_2(h) \leq \sqrt{d} \, M_2(h)$.\\

The next statement is our main result in the  multivariate setting.

\begin{theorem}\label{mtmd}
Under the above assumptions and notation, we have the following bounds:
\begin{enumerate}[{\normalfont (i)}]
 
 \item For every $g\in C^3(\R^d)$,  we have
 \begin{align}
  \babs{\E[g(F)]-\E[g(N)]}  & \leq     B_3(g) \,\, \sum_{i=1}^d\sqrt{  \E[F_i^4]-3\, \E[F_i^2]^2  } \notag  \\
  &    + A_2(g) \left( \sum_{i=1}^{d-1} \E[ F_i^4]^{1/4} \right) \sum_{j=2}^d \big( \E[F_j^4] -  3\E[F_j^2]^2 \big)^{1/4}   \,\, , \label{mdb2}
  \end{align}
  with ${\displaystyle B_3(g) = A_2(g)+ \frac{2q_d\sqrt{d \Tr(\Sigma)  }}{9 q_1}    M_3(g)     }$ and $A_2(g) = \dfrac{(2q_d-1)\sqrt{2d}}{4q_1}  M_2(g)$.\\
  
\item If  in addition $\Sigma$ is positive definite, then for every $g\in C^2(\R^d)$,   we have
\begin{align}
  \babs{\E[g(F)]-\E[g(N)]} &\leq  B_2(g)\,\, \sum_{i=1}^d\sqrt{  \E[F_i^4]-3\, \E[F_i^2]^2   } \notag \\
 &   + A_1(g) \left( \sum_{i=1}^{d-1} \E[ F_i^4]^{1/4} \right) \sum_{j=2}^d \big( \E[F_j^4] -  3\E[F_j^2]^2 \big)^{1/4}  \,\,,  \label{mdb4}
\end{align} 
with $$
 B_2(g) =   A_1(g)  + \frac{q_d\sqrt{2\pi }  \| \Sigma^{-1/2} \| _{\rm op}  \sqrt{\Tr(\Sigma)} }{6q_1}   M_2(g) 
 $$
 and
 $$
A_1(g) =   \dfrac{(2q_d-1)  \| \Sigma^{-1/2} \| _{\rm op}  }{q_1\sqrt{\pi}  }     M_1(g) \, .
$$
\end{enumerate}
\end{theorem}

\bigskip

The qualitative statement in the multivariate situation reads as follows.

\begin{cor}\label{cormd}
Fix $d\in\N$  and  $q_1,\dotsc,q_d\in\N$ and suppose that, for each $n\in\N$, $F^{(n)}:=(F^{(n)}_1,\dotsc,F^{(n)}_d)^T$  is a random vector such that each $F^{(n)}_k$ belongs to the $q_k$-th Poisson Wiener chaos. Moreover, assume that $C=C(i,j)_{1\leq i,j\leq d}$ is a fixed nonnegative definite matrix and that 
$N=(N_1,\dotsc,N_d)^T$ is a centred Gaussian vector with covariance matrix $C$. Assume that the following two conditions hold true:
\begin{enumerate}[{\normalfont (i)}]
 \item The covariance matrix of $F^{(n)}$ converges to $C$ as $n\to\infty$.
 \item For each $1\leq k\leq d$ it holds that $\lim_{n\to\infty}\E\bigl[(F^{(n)}_k)^4\bigr]=3 C(k,k)^2$.
\end{enumerate}
Then, as $n\to\infty$, the random vector $F^{(n)}$ converges in distribution to $N$.
\end{cor}

\begin{remark}\label{remmd}

 \begin{enumerate}[(a)]
 
 \item Comparing the bounds in Theorem \ref{mtmd} with the one provided in Theorem \ref{mt1d}, one observes that in the multivariate case the order of dependence on the fourth cumulants of the respective coordinates is $1/4$ instead of $1/2$.
 This phenomenon, which technically results from an application of the Cauchy-Schwarz inequality in order to disentangle certain joint moments of the coordinate variables, is nothing peculiar of the Poisson framework but also arises 
 in the Gaussian situation \cite{NN11} and in the recent multivariate de Jong type CLT for vectors of degenerate non-symmetric $U$-statistics \cite{DP16}. Moreover, this phenomenon only arises in the case when there are components belonging to different 
chaoses (see Remark \ref{rem43}). 

\item We stress that it is remarkable that, as in the Gaussian case \cite{PeTu}, the bounds and conditions in Theorem \ref{mtmd} and Corollary 
\ref{cormd} can be expressed just in terms of the individual fourth cumulants of the components of the random vectors. 
Indeed, both in the general situation of diffusive Markov generators (see \cite[Theorem 1.2]{CNPP}) and for the multivariate 
CLT for vectors of degenerate non-symmetric $U$-statistics (see \cite[Theorem 1.7]{DP16}), one additionally needs to assume the convergence of mixed fourth moments of those entries, which are of the same chaos and Hoeffding order, respectively.

 \item Corollary \ref{cormd} is a full Poisson space analogue  of the Peccati-Tudor theorem \cite{PeTu} for vectors of multiple integrals on a Gaussian space, which boils down the question about joint convergence of the whole vector 
 to conditions guaranteeing coordinatewise convergence (via Corollary \ref{cor1d}).  
\end{enumerate}

 \end{remark}

For the convenience of later reference, we state below the theorem of Peccati-Tudor on a Gaussian space.

\begin{theorem}[\cite{PeTu}] \label{PTudor-G}   Let $(W_t,t\in\R_+)$ be a  real standard Brownian motion, and we fix integers $d\geq 2$ and $1\leq q_1\leq \ldots \leq q_d$. Let $C =  C(i,j)_{1\leq i,j\leq d}$ be a $d\times d$ symmetric nonnegative definite matrix and for any $n\geq 1$, $i\in\{1,\ldots, d\}$, let $f_{n,i}\in L^2(\R_+^{q_i}, dx)$ be symmetric. Assume that the $d$-dimensional random vectors
\[
F^{(n)}  = \big( F^{(n)}_1, \ldots, F_d^{(n)} \big)^T : = \big(  I_{q_1}^W(f_{n,1}), \ldots, I_{q_d}^W(f_{n,d}) \big)^T
\]
satisfy 
$$
\lim_{n\to+\infty} \E\big[ F^{(n)}_i F^{(n)}_j \big] = C(i,j) \,,\quad i,j\in\{1,\ldots, d\} \, .
 $$
Then, as $n\to+\infty$, the following assertions are {\bf equivalent}:
\begin{enumerate}
\item[\rm (1)] The vector $F^{(n)} $ converges in distribution to a $d$-dimensional Gaussian vector $\mathcal{N}(0, C)$;

\item[\rm (2)] for every $i\in\{1,\ldots, d\}$, $F^{(n)}_i$ converges in distribution to a real Gaussian random variable $\mathcal{N}\big(0, C(i,i)\big)$;

\item[\rm (3)] for every $i\in\{1,\ldots, d\}$, $\E\big[ (F^{(n)}_i)^4\big]\to 3 C(i,i)^2$;

\item[\rm (4)] for every $i\in\{1,\ldots, d\}$ and each $1\leq r\leq q_i - 1$, $\big\| f_{n,i}\otimes_r f_{n,i} \big\| _{L^2(\R_+^{2q_i-2r})} \to 0$.

\end{enumerate}

\end{theorem}
For a proof, one can refer to \cite{NouPecbook}.
\bigskip

\subsection{Universality of Homogeneous sums}\label{univers}

The transfer principle in Proposition \ref{trans-p} is closely related to the {\it universality phenomenon} for homogeneous multilinear sums in independent Poisson random variables. We refer to the papers 
\cite{NPR-aop}, \cite{PZ2}, \cite{NPPS16} and \cite{BT16} for the universality results on homogeneous multilinear sums and it is worth pointing out that the reference \cite{NPPS16} also provides fourth moment theorems for homogeneous sums in general random variables satisfying  some moment conditions.  Before we can state the result, we   need to introduce some notation. 

\medskip

{\noindent \bf\small Notation.}  Suppose that $d\geq 2$, $N\in\N$, and that $f\in \ell^2(\N^d)$ is a function, which is symmetric in its arguments and vanishes on diagonals, {\it i.e.} for any $i_1, \ldots, i_d\in\N$,  $f(i_1, \ldots, i_d) = f(i_{\sigma(1)} , \ldots, i_{\sigma(d)}   ) $ for any $\sigma\in\mathbb{S}_d$ and $f(i_1, \ldots, i_d) = 0$, whenever $i_p = i_q$ for some $p\neq q$. For a sequence $\mathbf{X}=(X_i, i\in\N)$ of real random variables, we define the multilinear homogeneous sum of order $d$, based on the kernel $f$ and on the first $N$ elements of $\mathbf{X}$ by 
\begin{align}\label{sums}
 Q_d(f, N, \mathbf{X}) : & = \sum_{1\leq i_1, \ldots, i_d\leq N} f(i_1, \ldots, i_d) X_{i_1} \cdots X_{i_d}  \, .
\end{align}
Now let us consider an independent sequence $\mathbf{P}=(P_i, i\in\N)$ of normalised Poisson random variables, which can be  realised via our Poisson random measure $\eta$ on $\R_+$. More precisely, let $(t_i, i\in\N)$ be a strictly increasing sequence of  positive numbers. Set 
 $$P_i :=  \frac{\widehat{\eta}( [t_i, t_{i+1}) ) }{\sqrt{t_{i+1} - t_i}} = I_1^\eta \left( \frac{1}{  \sqrt{t_{i+1} - t_i}  } \mathbf{1}_{[t_i, t_{i+1})} \right) \,,$$
$i\in\N$.   We are now in the position to state the universality result.

\quad

\begin{theorem}\label{univ-thm} Let the above notation prevail.  Fix integers $d\geq 2$ and $q_d\geq \ldots \geq q_1\geq 2$. For each $j\in\{1,\dotsc,d\}$, let $(N_{n,j}, n\geq 1)$ be a sequence of natural numbers diverging to infinity, and  let $f_{n,j}: \{1, \ldots, N_{n,j  } \} ^{q_j}\to \R$ be symmetric and vanishing on diagonals such that 
\[
\lim_{n\to+\infty}\mathbf{1}_{(q_k = q_l)} q_k! \sum_{i_1, \ldots, i_{q_k} \leq N_{n,k}  }  f_{n, k}(i_1, \ldots, i_{q_k})  f_{n, l}(i_1, \ldots, i_{q_k})  = \Sigma(k,l)\, ,
\]
where $\Sigma = \Sigma(i,j)_{1\leq i,j\leq d}$ is a symmetric nonnegative definite $d$ by $d$ matrix.  Then the following condition $(A_0)$ implies the two {\bf equivalent} statements $(A_1)$, $(A_2)$ :

\begin{enumerate}

\item[$(A_0)$] For each $j\in\{1, \ldots, d\}$, one has 
${\displaystyle \lim_{n\to+\infty}  \E\big[   Q_{q_j}\big(f_{n,j}, N_{n,j}, \mathbf{P}\big)^4      \big] = 3 \Sigma(j,j)^2 \,. }$

\item[$(A_1)$] Let $\mathbf{G}$ be a sequence of i.i.d. standard Gaussian random variables, then, as $n\to+\infty$,
$\big(  Q_{q_1} (f_{n,1}, N_{n,1}, \mathbf{G} ), \ldots, Q_{q_d} (f_{n,d}, N_{n,d}, \mathbf{G} ) \big)^T$ converges in distribution to $\mathcal{N}(0, \Sigma)$.

\item[$(A_2)$] For every sequence $\mathbf{X} = \big( X_i, i\in\N\big)$ of independent centred random variables with unit variance and $\sup_{i\in\N} \E\big[ \vert X_i\vert^3 \big] < +\infty$, the sequence of $d$-dimensional random vectors  $\big(  Q_{q_1} (f_{n,1}, N_{n,1}, \mathbf{X} ), \ldots, Q_{q_d} (f_{n,d}, N_{n,d}, \mathbf{X} ) \big)^T$
converges in distribution to $\mathcal{N}(0, \Sigma)$, as $n\to+\infty$.

\end{enumerate}
If, in addition, $\inf \{ t_{i+1}-t_i \,: i\in\N  \} >0$, then $(A_0)$, $(A_1)$, $(A_2)$ are all equivalent and any of them is equivalent to the following assertion:
\begin{itemize}
\item[\rm $(A_3)$]
 $\big(  Q_{q_1} (f_{n,1}, N_{n,1}, \mathbf{P} ), \ldots, Q_{q_d} (f_{n,d}, N_{n,d}, \mathbf{P} )  \big)^T$
  converges  to $\mathcal{N}(0, \Sigma)$ in distribution,  as $n\to+\infty$.
\end{itemize}

\end{theorem}
\begin{remark}
The authors of  \cite{PZ2} established a fourth moment theorem for sequences of homogeneous sums in independent Poisson random variables whose variance is bounded away from zero, namely, $\inf \{ t_{i+1}-t_i \,: i\in\N  \} >0$ in our language. In particular, in order to get the implication ``$(A_0)\Rightarrow (A_1)$'', they relied heavily on the assumption that $\inf \{ t_{i+1}-t_i \,: i\in\N  \} >0$, which is inevitable due to their use of the product formula. As a consequence, our Theorem \ref{univ-thm} is an improvement of the results in \cite{PZ2}.    

\end{remark}

 \noindent{\bf Acknowledgement.}  We thank Ivan Nourdin and Giovanni Peccati for their helpful comments and stimulating discussions.  We also would like to thank a referee for several useful suggestions that allowed us to improve our work.

 \medskip
 
{\bf Plan of the paper.} In Section \ref{integrals}, we review some necessary definitions and facts about multiple integrals and Malliavin operators on the Poisson space. Section \ref{expairs} is devoted to the 
essential construction of a suitable family of exchangeable pairs for the concrete purpose of establishing fourth moment bounds on the Poisson space.  
In order to make use of it, we also state two new abstract plug-in results for such families of exchangeable pairs. In Section \ref{proofmd} we give the proofs of our main results, whereas Section \ref{prooftech} presents the proofs of Proposition \ref{trans-p}, Theorem \ref{univ-thm} as well as   certain technical auxiliary results.

\bigskip



\section{Some stochastic analysis on the Poisson space}\label{integrals}

\bigskip

\subsection{Basic operators and notation}
For a positive integer $p$, we denote by $L^2(\mu^p)$ the Hilbert space of all square-integrable and real-valued functions on $\mathcal{Z}^p$, and we denote by $L^2_s(\mu^p)$ the subspace of $L^2(\mu^p)$ whose elements are $\mu^p$-a.e. symmetric. Moreover, we indicate by $\Enorm{\cdot}$ and $\langle \cdot,\cdot\rangle_2$ respectively the usual norm and scalar product 
on $L^2(\mu^p)$ for any value of $p$. We also set $L^2(\mu^0):=\R$. For $f\in L^2(\mu^p)$, we define $I^\eta_p(f)$ to be the \textit{multiple Wiener-It\^o integral} of $f$ with respect to the compensated Poisson random measure $\widehat{\eta}$. 
If $p=0$, then, by convention, $I^\eta_0(c):=c$ for each $c\in\R$.

The multiple Wiener-It\^o integrals satisfy the following properties:

\begin{enumerate}[1)]
 \item For $p\in\N$ and $f\in L^2(\mu^p)$,  $I^\eta_p(f)=I^\eta_p(\widetilde{f})$, where $\widetilde{f}$ denotes the \textit{symmetrization} of $f\in L^2(\mu^p)$, i.e. 
  \[
  \widetilde{f}(z_1,\dotsc,z_p)=\frac{1}{p!}\sum_{\pi\in\mathbb{S}_p} f(z_{\pi(1)},\dotsc,z_{\pi(p)})\,\,,
  \]
where  $\mathbb{S}_p$ is the symmetric group acting on $\{1,\dotsc,p\}$. Note $\widetilde{c} = c$ for any $c\in\R$.
\item For $p,q\in\N_0$ and $f\in L^2(\mu^p)$, $g\in L^2(\mu^q)$, one has $I^\eta_p(f),\,I^\eta_q(g)\in L^2(\Prob)$ and $\E\bigl[ I^\eta_p(f)I^\eta_q(g)\bigr]= \delta_{p,q}\,p!\,\langle \widetilde{f},\widetilde{g}\rangle_2 $, where $\delta_{p,q}$ denotes \textit{Kronecker's delta symbol}.
\end{enumerate}
See Section 3 of \cite{Lastsv} for the proofs of the above well known results.           

\bigskip
           
For $p\in \N_0$, the Hilbert space $C_p:=\{I^\eta_p(f), \, f\in L^2(\mu^p)\}$, is called the \textit{$p$-th Poisson Wiener chaos} associated with $\eta$. The well-known {\it Wiener-It\^o chaotic decomposition} states that every $F\in L^2(\Prob): = L^2(\Omega, \sigma\{\eta\}, \Prob)$ admits 
a unique representation 
\begin{equation}\label{chaosdec}
 F=\E[F]+\sum_{p=1}^\infty I^\eta_p(f_p)\,\,\,\,\text{in $L^2(\Prob)$, where $f_p\in L_s^2(\mu^p)$, $p\geq1$.  }
\end{equation}

Let $F\in L^2(\P)$ and $p\in\N_0$, then we define by $J_p(F)$ the orthogonal projection of $F$ on $C_p$. Note that, if $F$ has the chaotic decomposition as in \eqref{chaosdec}, then $J_p(F)=I^\eta_p(f_p)$ for all $p\geq1$ and $J_0(F) = \E[F]$.  

For $F\in L^2(\Prob)$ with the chaotic decomposition as in \eqref{chaosdec}, we define 
 $$P_tF = \E[ F]+ \sum_{p\geq 1}e^{-pt}\, I^\eta_p(f_p) \,\, . $$
This gives us the \textit{Ornstein-Uhlenbeck semigroup} $(P_t, t\in\R_+)$.  The domain $\dom L$ of the \textit{Ornstein-Uhlenbeck generator} $L$ is the set of those  $F\in L^2(\Prob)$ with the chaotic decomposition  \eqref{chaosdec} verifing $\sum_{p=1}^\infty p^2\,p!\Enorm{f_p}^2 < + \infty$, and for $F\in\dom L$, one has
\begin{equation}\label{defL}
 LF=-\sum_{p=1}^\infty p I^\eta_p(f_p)\,.
\end{equation}
We conclude from \eqref{defL} that $LF$ is always centred, $\N_0$ is  the spectrum of $-L$ and   $F\in \dom L$ is an eigenfunction of $-L$ with corresponding eigenvalue $p$ if and only if $F=I^\eta_p(f_p)$ for some $f_p\in L^2_s(\mu^p)$, {\it i.e.}  $C_p={\rm Ker}(L+pI)$.

Moreover, it is easy to see that $L$ is \textit{symmetric} in the sense that  
$
\E[GLF]=\E[FLG]
$
for all $F,G\in\dom L$. Finally, for $F,G\in\dom L$ with $FG\in\dom L$, we define the \textit{carr\'{e} du champ} operator $\Gamma$ associated with $L$ by 
\begin{equation}\label{cdc}
 \Gamma(F,G):=\frac{1}{2}\bigl(L(FG)-FLG-GLF\bigr)\,,
\end{equation}
and it is easy to verify that $\E [ \Gamma(F, G)  ] = - \E [ FLG  ] = - \E [ GLF  ]$.   
It follows from Lemma \ref{ledp1} below that $\Gamma(F,G)$ is always well-defined whenever $F,G\in L^4(\P)$ and both have a finite chaotic decomposition.  

In the book 
  \cite{BD-book}, the authors develop Dirichlet form method for the Poisson point process, and  starting from the Dirichlet form associated with the Ornstein-Uhlenbeck structure, they obtain an expression of carr\'e du champs operator that is close to the one derived in \cite{DP17}.  As readers will see, we only need the    spectral decomposition rather than the intrinsic tools in \cite{BD-book, DP17}.  This   highlights the elementary feature of our method.

 \bigskip

For  $p,q\in\N$, $0\leq r\leq p\wedge q$ and $f\in L_s^2(\mu^p)$ and $g\in L_s^2(\mu^q)$, we define the $r$-th \textit{contraction} $f\otimes_r g:\mathcal{Z}^{p+q-2r}\rightarrow\R$ by
\begin{align*}
 f\otimes_r g(x_1,\dotsc,x_{p-r}, y_1,\dotsc,y_{q-r})&:=\int_{\mathcal{Z}^r} f(x_1,\dotsc,x_{p-r},z_1,\dotsc,z_r)\\
 &\hspace{1.5cm}\cdot g(y_1,\dotsc,y_{q-r},z_1,\dotsc,z_r)d\mu^r(z_1,\dotsc,z_r)\,.
\end{align*}
Observe that $f\otimes_r g\in L^2(\mu^{p+q-2r})$  is  in general not symmetric and  that $f\otimes_0 g = f\otimes g$ is simply the tensor product of $f$ and $g$.

\begin{lemma}[Lemma 2.4 of \cite{DP17}]\label{ledp1}
 Let $p,q\in\N$ and  $F=I^\eta_p(f)$, $G=I^\eta_q(g)$ be in $L^4(\Prob)$ with $f,g$ symmetric, then
 $FG$ has a finite chaotic decomposition of the form
  $$
  FG=\sum_{r=0}^{p+q}J_r(FG)=\sum_{r=0}^{p+q}I^\eta_r(h_r)\,,
  $$
 where $h_r\in L_s^2(\mu^r)$ for each $r$.  In particular, $h_{p+q}=f\widetilde{\otimes} g$.\\
\end{lemma}

\subsection{Useful estimates via spectral decomposition}

To conclude  the section, we state several lemmas that are useful for our proofs.    

\begin{lemma}\label{leint1}
 Let $F\in L^4(\P)\cap C_p$ and  $G\in L^4(\P)\cap C_q$ for $p,q\in\N$.  Then, 
 \begin{equation}
  \Var\bigl(\Gamma(F,G)\bigr)\leq\frac{(p+q-1)^2}{4}\bigl(\E\bigl[F^2G^2\bigr]-2\,  \E[FG]^2-\Var(F)\Var(G)\bigr) \,\, , \label{DP-0}
 \end{equation}
and
   \begin{equation}
   0\leq\frac{3}{p}\E\bigl[F^2\Gamma(F,F)\bigr]-\E\bigl[F^4\bigr]\leq\frac{4p-3}{2p}\bigl(\E\bigl[F^4\bigr]-3\E\bigl[F^2\bigr]^2\bigr) \,\,. \label{DP3.2}
  \end{equation}
In particular, for $F=G$, we obtain 
\begin{equation}
 \Var\bigl(\Gamma(F,F)\bigr)\leq\frac{(2p-1)^2}{4}\bigl(\E[F^4]-3\E[F^2]^2\bigr)  \, . \label{DP3.1}   
\end{equation}
\end{lemma}

Note that the authors of \cite{DP17} provided a proof of    \eqref{DP3.2} under the {\bf Assumption A} therein, while we only require the assumption of finite fourth moment.  Although \eqref{DP3.1} is the content of  Lemma 3.1 in \cite{DP17}, we will provide another proof, in which we deduce a nice relation between contractions of kernels and the fourth cumulant. Such a relation is crucial for us to obtain the  transfer principle ``from-Poisson-to-Gaussian''.  The proof of Lemma \ref{leint1} as well as  that of the next lemma will be presented in Section 5.

\begin{lemma}\label{leint2}
Under the same assumptions of Lemma \ref{leint1}, we have that 
\begin{itemize}
\item[\rm(1)] If $p < q$, then
\begin{align}
{\rm Cov}(F^2, G^2) &= \E\big[ F^2 G^2\big] -  {\rm Var}(F){\rm Var}(G) \notag\\
& \leq \sqrt{\E[ F^4]} \sqrt{\E[ G^4] - 3 \E[G^2]^2 } \,\,; \label{bel1}
 \end{align}
 \item[\rm(2)] if $p=q$, then
\begin{align}
\qquad{\rm Cov}(F^2, G^2)  - 2\, \E[ FG ]^2  \leq 2  \sqrt{\big(  \E[F^4] - 3\, \E[F^2]^2 \big)   \big(  \E[G^4] - 3 \E[G^2]^2 \big)  } \, . \label{bel2}
\end{align} 
\end{itemize}

\end{lemma}
This lemma is motivated by Proposition 3.6 in \cite{CNPP}.

\bigskip

\section{Stein's method of exchangeable pairs}\label{expairs}   
 
 \bigskip

The exchangeable pairs approach within Stein's method was first used in the paper \cite{Dia77} which, however, attributes the method to Charles  Stein himself. Later, this technique was presented in a systematic way in Stein's monograph  
\cite{S:86}. We recall that a pair $(X,X')$ of random elements on a common probability space is said to be \textit{exchangeable},
 if $(X,X')$ has the same distribution as $(X',X)$. In the book \cite{S:86}, it is 
highlighted that a given real random variable $W$ is  close in distribution to a standard normal variable $N$, whenever one can construct an exchangeable pair $(W,W')$ such that $W'$ is close to $W$ in {\it some} sense and that
the \textit{linear regression property}
\begin{equation*}
 \E\bigl[W'-W\,\bigl|\,W\bigr]=-\lambda W
\end{equation*}
is satisfied for some  small  $\lambda>0$  and $\Var\big(\frac{1}{2\lambda} \E[ (W' - W)^2\vert W ] \big)$ is small.  For a precise statement,   we refer to 
\cite[Theorem 3.1]{S:86}. 

In recent years, the method of exchangeable pairs has been generalised for other distributions and  multi-dimensional settings in many papers  like 
\cite{RiRo97,Ro08,ShSu,ReiRoe09,ChaMe08,ChSh,CFR11,DoeBeta,EiLo10}, to name   a few.

Moreover, the articles \cite{Meckes06,ChaMe08,Meck09,DoeSto11} develop versions of the exchangeable pairs method suitable for situations, in which one can construct a continuous family $(W,W_t)_{t>0}$ of exchangeable pairs.
By their continuity assumptions, these papers   succeed in reducing the order of smoothness of  test functions  and hence in obtaining bounds in more sophisticated probabilistic distances. For instance, the bounds from \cite{Meckes06} are expressed in terms of the total variation distance. It is this framework of exchangeable pairs that is most closely related to the variant of the method developed in the present paper. 
In contrast to the quoted papers, however, our abstract results on exchangeable pairs do not make such strong continuity assumptions and hence, allow us to deal with the inherent discreteness of the Poisson space, which, in general,  does not even allow for convergence in the total variation distance, see Section \ref{plugin} for details.

\subsection{Exchangeable pairs  constructed via continuous thinning}

   Recall that, in our general framework, $\eta$ is a Poisson random measure on some  $\sigma$-finite measure space $\big( \mathcal{Z}, \mathscr{Z}, \mu \big)$.  As a consequence, we can assume that $\eta$ is a {\it proper} Poisson point process, that is,  almost surely 
 \begin{align} \label{form}
 \eta = \sum_{n=1}^\kappa \delta_{X_n} \,\,, 
 \end{align}
 where $X_n$, $n\geq 1$ are random variables with values in $\mathcal{Z}$ and $\kappa$ is a $\N_0\cup\{+\infty\}$-valued random variable. Indeed, according to Corollary 3.7 in \cite{LPbook}, any  Poisson random measure $\eta$ on some $\sigma$-finite measure space is equal in distribution to some proper Poisson point process. As in this work, we are only concerned with distributional properties, we will always assume that $\eta$ is of the form \eqref{form}.

  Let $\mathbf{N}_\sigma$ be the collection of $\sigma$-finite measures $\nu: \mathscr{Z}\to \N_0\cup\{+\infty\}$ and $\mathscr{N}_\sigma(\mathcal{Z})$ be the $\sigma$-algebra generated by the maps $\nu\in\mathbf{N}_\sigma\longmapsto \nu(B)$, $B\in\mathscr{Z}$. We consider the Poisson point process $\eta$ as a random element in $\big(\mathbf{N}_\sigma, \mathscr{N}_\sigma(\mathcal{Z})\big)$. Moreover, for any $F\in L^0(\Omega, \sigma\{\eta\}, \Prob)$,  one can find   a ($\Prob$-a.s. unique) {\it representative}   $\mathfrak{f}$ of $F$ such that $F = \mathfrak{f}(\eta)$,  see \cite{Lastsv} for more details.

 Now let $\mathbb{Q}$ be a standard exponential measure on $\R_+$ with density $\exp(-y)\, dy$, and let $(Y_n)_{n\in\N}$ be a sequence of {\it i.i.d.} random variables with distribution $\mathbb{Q}$, independent of  $(\kappa, X_n)$.  Then the {\it marked} point process $\xi$, given by 
$$\xi := \sum_{n=1}^\kappa \delta_{(X_n,  Y_n)}\,\,,$$
 is a Poisson point process with control   $\mu\otimes \mathbb{Q}$. For each $t\in\R_+$, we define 
 \[
 \eta_{e^{-t}}(A) := \xi\big(  A\times[t, +\infty) \big) \,\,,
 \]
 which is called  the $e^{-t}$-thinning of $\eta$: it is obtained by    removing  the atoms $(X_n)$ in $\eta$  independently of each other with probability $1 - e^{-t}$.  Moreover, $\eta_{e^{-t}}$ and $\eta - \eta_{e^{-t}}$ are two independent Poisson point processes with control measure $e^{-t}\mu$, $(1- e^{-t})\mu$ respectively.  One can refer to Chapter 5 in \cite{LPbook} for more details.

 For any fixed $t\geq 0$, let $\eta'_{1-e^{-t}}$ be a Poisson point process on $\mathcal{Z}$ with control $(1- e^{-t})\mu$ such that it is independent of $(\eta, \eta_{e^{-t}})$. Then the \textit{Mehler formula} gives a useful representation of the Ornstein-Uhlenbeck semigroup $(P_t)$: for  $F\in L^2(\Omega, \sigma\{\eta\}, \Prob)$, 
 $$P_tF = \E\big[ \mathfrak{f}(\eta_{e^{-t}} + \eta'_{1-e^{-t}} ) \big\vert \sigma\{\eta\} \big] \,\, , $$
 where $\mathfrak{f}$ is a {\it representative} of $F$, see \cite{Lastsv} for more details.     We remark that the Mehler formula on the Poisson space has already been effectively used in \cite{LPS} in order   to obtain a pathwise representation for the pseudo-inverse of the Ornstein-Uhlenbeck generator $L$ on the Poisson space,  which  has led to second-order Poincar\'e inequalities.
 \bigskip
 
 We record    an important observation in the following lemma.
 
 \begin{lemma}\label{OB}  For each $t\in\R_+$,  set $\eta^t := \eta_{e^{-t}} + \eta'_{1-e^{-t}}$.  Then,
     $\big(\eta,  \eta^t\big)$ is an exchangeable pair of Poisson point processes.
 \end{lemma} 

\begin{proof}

 To prove this lemma, it suffices to notice that $\eta = \eta_{e^{-t}} + \eta -\eta_{e^{-t}}$ and  that $\eta -\eta_{e^{-t}}$, $\eta'_{1-e^{-t}}$ have the same law, and are  both independent of $\eta_{e^{-t}}$.
 
 Let $\mathfrak{f}: \mathbf{N}_\sigma\to \R$ be $\mathscr{N}_\sigma(\mathcal{Z})$-measurable; then, for any Borel subsets $A_1, A_2$ of $\R$, one has that
 \begin{align*}
&\quad \Prob\big( \mathfrak{f}(\eta)\in A_1\,, \,\, \mathfrak{f}(\eta^t)\in A_2 \big) \\
& =  \Prob\big( \mathfrak{f}(\eta_{e^{-t}}   + \eta - \eta_{e^{-t}} )\in A_1\,, \,\, \mathfrak{f}(\eta_{e^{-t}} + \eta'_{1- e^{-t}})\in A_2 \big) \\
 &  = \Prob\big( \mathfrak{f}(\eta_{e^{-t}}   + \eta'_{1- e^{-t}} )\in A_1\,, \,\, \mathfrak{f}(\eta_{e^{-t}} + \eta - \eta_{e^{-t}} )\in A_2 \big) \quad\text{(by conditioning on $\eta_{e^{-t}}$)} \\
 & =  \Prob\big( \mathfrak{f}(\eta^t)\in A_1\,, \,\, \mathfrak{f}(\eta)\in A_2 \big) \,\, .
 \end{align*}
 This implies the exchangeability of $(\eta, \eta^t)$. \qedhere

\end{proof}  
 
 \bigskip
 The following  result is  a consequence of Lemma \ref{OB} and Mehler formula: it is a key ingredient for us to obtain   exact fourth moment theorems in any dimension.  Indeed, it fits extremely well with the abstract results for exchangeable pairs that are presented in Section 3.2.

 \begin{prop}\label{GAMMA}    Let $F= I_q^\eta(f)\in L^4(\Prob)$ for some $f\in L_s^2(\mu^q)$ and     define $F_t = I^{\eta^t}_q(f)$. Then, $(F, F_t)$ is  an exchangeable pair for each $t\in\R_+$. Moreover,
\begin{itemize}
 \item[\rm (a)] ${\displaystyle  \lim_{t\downarrow 0}  \frac{1}{t} \E\big[ F_t - F \vert \sigma\{ \eta\} \big] =  L F = -qF}$ in $L^4(\Prob)$.
 
 \item[\rm (b)] If $G = I_p^\eta(g)\in L^4(\Prob)$ and $G_t =I^{\eta^t}_p(g) $ for  some $g\in L_s^2(\mu^p)$, then we have     ${\displaystyle  \lim_{t\downarrow 0}  \frac{1}{t} \E\big[ (F_t - F)(G_t - G) \vert \sigma\{ \eta\} \big] =  2 \Gamma(F,G)}$, with the convergence    in $L^2(\Prob)$. 
 
 \item[\rm (c)]   ${\displaystyle  \lim_{t\downarrow 0}  \frac{1}{t} \E\big[ (F_t - F)^4 \big] = -4q \,\E[F^4] + 12 \, \E\big[ F^2 \Gamma(F,F) \big] \geq 0}$.

\end{itemize}

\end{prop}

\begin{proof}   The exchangeability of $F, F_t$ is an immediate consequence of Lemma \ref{OB}.     Relation (a) is a direct consequence of the Mehler formula:
\begin{align*}
  \frac{1}{t} \E\big[ F_t - F \vert \sigma\{ \eta\} \big]    =  \frac{P_t(F) - F}{t}  = \frac{e^{-qt} - 1}{t} F \, ,
  \end{align*}
 and such a quantity converges almost surely,  and in $L^4(\Prob)$ to $LF = -qF$, as $t\downarrow 0$.

  By Lemma \ref{ledp1}, $FG = \sum_{k=0}^{p+q} J_k(FG) = \sum_{k=0}^{p+q} I_k^{\eta}(h_k) $ for some  $h_k\in L^2_s(\mu^k)$, and consequently $F_tG_t = \sum_{k=0}^{p+q} I_k^{\eta^t}(h_k)$, so that 
$$ \frac{1}{t} \E\big[ F_tG_t -FG \vert \sigma\{ \eta\} \big]  =\frac{1}{t}  \sum_{k=0}^{p+q} \E\big[ I_k^{\eta^t}(h_k) -  I_k^{\eta}(h_k) \vert \sigma\{\eta\} \big]  $$
converges almost surely and  in $L^2(\Prob)$ to  $\sum_{k=0}^{p+q} -k \, J_k(FG) = L(FG)$, as $t\downarrow 0$.  Hence  almost surely and in  $L^2(\Prob)$, we infer that
\begin{eqnarray*}
  && \frac{1}{t} \E\big[ (F_t - F)(G_t - G) \vert \sigma\{ \eta\} \big]   \\
  &=&  \frac{1}{t} \E\big[ F_tG_t -FG \vert \sigma\{ \eta\} \big]  -  F\frac{\E [ G_t -G \vert \sigma\{ \eta\}  ] }{t}  -  G\frac{\E [ F_t -F \vert \sigma\{ \eta\}  ] }{t} \\
  &\to& L(FG) - FLG - GLF = 2\, \Gamma(F,G) \, ,
 \end{eqnarray*}
as $t\downarrow 0$.  Since the pair $(F,F_t)$ is exchangeable, we can write
    \begin{eqnarray}
  \E\big[ (F_t - F)^4 \big] & =&   \E\big[ F_t^4 + F^4 - 4 F_t^3F - 4 F^3 F_t + 6 F_t^2 F^2   \big]  \notag\\
  & = &2 \E[ F^4] - 8 \E\big[ F^3 F_t \big] +  6 \E\big[ F^2 F_t^2    \big]  \quad\text{(by exchangeability)} \notag\\
  & = &4 \E\big[ F^3(F_t - F) \big]  +6 \E\big[ F^2   (F_t - F)^2 \big] \quad\text{(after rearrangement)}\notag \\
  & = &4 \E\big[ F^3 \E[ F_t - F \vert \sigma\{ \eta \} ] \big]  +6 \E\big[ F^2 \E[   (F_t - F)^2 \vert \sigma\{ \eta \} ] \big].\notag
    \end{eqnarray}   
so (c) follows immediately from  (a),(b) and the fact that $F\in L^4(\Prob)$.            \end{proof}

\bigskip

\subsection{Abstract results for exchangeable pairs}\label{plugin}

As indicated in the introductory part of this section, the following two Propositions should be seen as complements to \cite[Theorem 1.4]{Meckes06} and \cite[Theorem 4]{Meck09} as well as \cite[Theorem 2.4]{ChaMe08}, respectively.  
The main difference with respect to these results, as mentioned above, is that we do not assume any continuity from the respective families of exchangeable pairs, which precisely means that we allow for non-zero limits in the respective conditions (c) below. 

 \begin{prop}\label{Meckes06}  Let $Y$ and a family of  random variables $(Y_t)_{t\geq 0}$   be defined on a common probability space $(\Omega,\mathcal{F}, \Prob)$ such that $Y_t\overset{law}{=} Y$ for every $t\geq 0$. Assume that $Y\in L^4(\Omega, \mathscr{G}, \Prob)$ for some $\sigma$-algebra $\mathscr{G}\subset\mathcal{F}$  and that, in $L^1(\Prob)$,
\begin{enumerate}
\item[\rm (a)] ${\displaystyle \lim_{t\downarrow 0} \frac1t\,\E[Y_t-Y  |\mathscr{G}] = -\lambda\,Y}$ for some $\lambda>0$,
\item[\rm (b)] ${\displaystyle   \lim_{t\downarrow 0} \frac1t\,\E[(Y_t-Y)^2|\mathscr{G}] = (2\lambda+S){\rm Var}(Y) }$ for some random variable $S$,
\item[\rm (c)]   ${\displaystyle  \lim_{t\downarrow 0} \frac{1}{t}\,\E\big[(Y_t-Y)^4\big]= \rho(Y)\Var(Y)^2}$ for some  $\rho(Y) \geq 0$.
\end{enumerate}
Then, with $N\sim \mathcal{N}(0,{\rm Var}(Y))$, we have  
$$ d_\W(Y,N) \leq   \frac{\sqrt{{\rm Var}(Y)}}{\lambda\sqrt{2\pi}} \E\big[ |S| \big] + \frac{ \sqrt{  ( 2\lambda+\E[S]) {\rm Var}(Y)       }}{3\lambda}   \sqrt{\rho(Y)} \,\,.   $$

\end{prop}

\begin{remark}     If the quantity $\rho(Y) = 0$ in {\rm (c)}, then Proposition \ref{Meckes06} reduces to Theorem 1.3 in \cite{NZ17} and one has 
\[
 d_{\rm TV}\big( Y,   N \big) : = \sup_{A\subset \R\,\,\text{Borel}} \big\vert  \Prob(Y\in A) -  \Prob( N\in A)  \big\vert \leq \frac{\E\big[ \vert S \vert \big]}{\lambda} \,\, .
\]

\end{remark}

The following result is a multivariate extension of Proposition \ref{Meckes06}. The proofs will be postponed to Section 5.5 and 5.6.

\begin{prop}\label{Meckes09} 

For each $t > 0$, let $(X,X_t)$ be an exchangeable pair of centred  $d$-dimensional random vectors defined on a common probability space $(\Omega, \mathcal{F}, \Prob)$. Let $\mathscr{G}$ be a $\sigma$-algebra that contains $\sigma\{X\}$.  Assume that $\Lambda\in\R^{d\times d}$ is an invertible deterministic matrix and $\Sigma$ is a symmetric, non-negative definite deterministic matrix such that 
\begin{enumerate}
\item[\rm (a)]
${ \displaystyle \lim_{t\downarrow 0} \frac{1}{t}\,  \E\big[ X_t - X |\mathscr{G} \big] =  - \Lambda X}$ in $L^1(\Prob)$,
\item[\rm (b)]
${\displaystyle
 \lim_{t\downarrow 0} \frac{1}{t}\, \E\big[ ( X_t - X)(X_t- X)^T |\mathscr{G} \big] =  2\Lambda \Sigma + S}$ in $L^1(\Omega, \| \cdot \| _{\rm H.S.})$ for some random  matrix $S$,
 \item[\rm (c)]  for each $i\in\{1,\ldots, d\}$, there exists some real number $\rho_i(X)$ such that
$$\lim_{t\downarrow 0}   \frac{1}{t}\, \E\big[ ( X_{i,t} - X_i)^4 \big] =  \rho_i(X) \,\, ,$$ 
where $X_{i,t}$ (resp. $X_i$) stands for the $i$-th coordinate of $X_t$ (resp. $X$).
\end{enumerate}
Then, with $N\sim \mathcal{N}(0,\Sigma)$, we have the following bounds:
\begin{itemize}
\item[\rm (1)] For $g\in C^3(\R^d)$ such that $g(X), g(N)\in L^1(\Prob)$,  one has
\begin{flalign*}
 &\quad \Big\vert  \E  \big[ g(X)  - g(N)  \big]  \Big\vert   \\
 &\leq  \Theta_1(g)\, \E[ \| S \| _{\rm H.S.} ] +\Theta_2(g)\, \sqrt{\sum_{i=1}^d 2\Lambda_{i,i} \Sigma_{i,i} + \E[S_{i,i} ] }\sqrt{\sum_{i=1}^d\rho_i(X)} \,,
\end{flalign*}
where the constants $\Theta_1(g)$ and $\Theta_2(g)$ are given by 
 \begin{flalign}\label{Theta12}
 \quad\Theta_1(g) = \frac{ \| \Lambda^{-1} \| _{\rm op} \, M_2(g) \sqrt{d}  }{4}  \quad \text{ and}\quad \Theta_2(g) =  \frac{\sqrt{d}  M_3(g) \| \Lambda^{-1} \| _{\rm op} }{18}  \,\, .
 \end{flalign}
 \item[\rm (2)] If, in addition, $\Sigma$ is positive definite, then for $g\in C^2(\R^d)$ such that  \\$g(X), g(N)\in L^1(\Prob)$, one has
\begin{flalign*}
 &\quad \Big\vert  \E  \big[ g(X)  - g(N)  \big]  \Big\vert    \\
 &\leq  K_1(g)\, \E[ \| S \| _{\rm H.S.} ] +  K_2(g) \,\, \sqrt{\sum_{i=1}^d 2\Lambda_{i,i} \Sigma_{i,i} + \E[S_{i,i} ] }\sqrt{\sum_{i=1}^d\rho_i(X)} \,\, ,
 \end{flalign*}

 where the constants $K_1(g)$ and $K_2(g)$ are given by 
 \begin{align}\label{K12}
 K_1(g) & = \dfrac{M_1(g) \| \Lambda^{-1} \| _{\rm op}  \|  \Sigma^{-1/2} \| _{\rm op} }{\sqrt{2\pi}} \,,\, \\
 K_2(g) &= \dfrac{\sqrt{2\pi}     M_2(g) \| \Lambda^{-1}\| _{\rm op} \|  \Sigma^{-1/2} \| _{\rm op} }{24} \, .
 \end{align}
 \end{itemize}
 \end{prop}

\bigskip

\section{Proofs of main results}\label{proofmd}

\bigskip

\subsection{Proof of Theorem \ref{mt1d}}\label{proof1d}    Without loss of    generality, we assume $F = I_q^\eta(f)$ for some $f\in L^2_s(\mu^q)$, and we define $F_t =  I_q^{\eta^t}(f)$ for $t\in\R_+$.
Then, by Proposition \ref{GAMMA}, $(F, F_t)$ is an exchangeable pair and the assumptions (a), (b), (c) in Proposition \ref{Meckes06} are satistified with 
\begin{itemize}
\item $\lambda = q $ \qquad $\bullet$ $S = 2\dfrac{\Gamma(F,F)}{\sigma^2} - 2q$    \qquad $\bullet$ $\rho(F) = \dfrac{-4q \E[ F^4] + 12 \E[ F^2 \Gamma(F,F) ]}{\sigma^4}$.

\end{itemize}
More precisely, 
\begin{enumerate}
\item[\rm (a)] ${\displaystyle \lim_{t\downarrow 0} \frac{1}{t}\,\E[F_t-F  |  \sigma\{\eta\}  ] = -q F}$,
\item[\rm (b)] ${\displaystyle   \lim_{t\downarrow 0} \frac{1}{t}\,\E[(F_t-F)^2| \sigma\{\eta\}    ] = 2\Gamma(F, F)\, }$,
\item[\rm (c)]   ${\displaystyle  \lim_{t\downarrow 0} \frac{1}{t}\,\E\big[(F_t-F)^4\big]= \rho(F)\sigma^4}$.
\end{enumerate}

Therefore, one has (using that $\E\big[ \Gamma(F,F) \big] = q\, \E[ F^2] $ )
\begin{align*}
d_\W\big(F, \mathcal{N}(0, \sigma^2) \big) & \leq  \frac{\sqrt{2/\pi}}{\sigma q} \sqrt{ {\rm Var}\big(\Gamma(F,F) \big)  } + \frac{2\sqrt{2} }{3\sigma} \sqrt{ \frac{3}{q} \E\big[  F^2 \Gamma( F,F )   \big] - \E[ F^4]  } \,\, .
\end{align*}
The desired result follows immediately from Lemma \ref{leint1}.

\bigskip

\subsection{Proof of Theorem \ref{mtmd}}   Assume that 
$$F = \big( F_1, \ldots, F_d\big)^T = \big( I_{q_1}^\eta(f_1) , \ldots, I_{q_d}^\eta(f_d) \big)^T $$
with $1\leq q_1\leq \ldots \leq q_d$ and $f_j\in L^2_s(\mu^{q_j})$ for each $j$, and   for each $t\in\R_+$, define
$$F_t = \big( F_{1,t}, \ldots, F_{d,t}\big)^T = \big( I_{q_1}^{\eta^t}(f_1) , \ldots, I_{q_d}^{\eta^t}(f_d) \big)^T\,\,. $$
Then, by Lemma \ref{OB}, $(F_t, F)$ is an exchangeable pair and by Proposition \ref{GAMMA}, we deduce 
\[ 
 \E\left[\, \frac{1}{t}(F_{i,t} - F_i)(F_{j,t} - F_j) - 2 \Gamma(F_i, F_j) \, \big\vert\,\, \sigma\{\eta\} \,\right] \to 0 \,\,,
\]
as $t\downarrow 0$, where the convergence takes place in $L^2(\Prob)$.  Therefore, as $t\downarrow 0$ and in $L^1(\Prob)$, we have 
\begin{align*}
&\quad \left\|  \frac{1}{t} \E\big[ (F_t - F)(F_t - F)^T \vert \sigma\{ \eta\} \big] - \big(   2\Gamma(F_i, F_j) \big)_{1\leq i,j\leq d}       \right\| _{\rm H.S.}^2 \\
  & = \sum_{i,j=1}^d  \left( \, \E\left[\, \frac{1}{t}(F_{i,t} - F_i)(F_{j,t} - F_j) - 2 \Gamma(F_i, F_j) \, \big\vert\,\, \sigma\{\eta\} \,\right] \,\right)^2  \to 0 \, \, .
\end{align*}
It is easy to see that for each $j\in\{1, \ldots, d\}$, 
\[
\lim_{t\downarrow 0} \frac{1}{t}\,  \E\big[ F_{j,t} - F_j |\sigma\{\eta\} \big] =  - q_j F_j \quad\text{in $L^4(\Prob)$,}
\]
from which we deduce that as $t\downarrow 0$ and in $L^2(\Prob)$, we have  
\begin{align*}
 \left\| \,\,  \frac{1}{t} \E\big[ F_t - F \vert \sigma\{\eta\} \big] - \Lambda F \right\| _2^2  = \sum_{j=1}^d \left(\,  \E\left[ \frac{F_{j,t} - F_j }{t}  + q_j F_j \big\vert \, \sigma\{\eta\} \right]  \right)^2  \to 0 \,\, ,
\end{align*}
with $\Lambda = \text{diag}(q_1, \ldots, q_d)$ in such a way that  $\Opnorm{\Lambda^{-1}} = 1/q_1$.

It is also clear that, for each $i\in\{1, \ldots, d \}$,
\begin{align*}
\rho_i(F) : =  \lim_{t\downarrow 0}   \frac{1}{t}\, \E\big[ ( F_{i,t} - F_i)^4 \big] & = -4q_i \, \E[ F_i^4] + 12 \E\big[ F_i^2 \Gamma(F_i,F_i) \big]  \\
&\leq 2(4q_i-3) \Big( \E[ F_i^4] - 3 \E[ F_i^2]^2 \Big) \quad\text{by \eqref{DP3.2}.}
\end{align*}
Now   define $S_{i,j} := 2\Gamma(F_i,F_j) - 2q_i\, \Sigma_{i,j}$ for $i,j\in\{1, \ldots, d\}$, and observe in particular that $S_{i,j}$ has zero mean.   Thus, 
 \begin{align}
&\quad   \sqrt{\sum_{i=1}^d 2\Lambda_{i,i} \Sigma_{i,i} + \E[S_{i,i} ] }\sqrt{\sum_{i=1}^d\rho_i(F)} \notag \\
  & \leq  \sqrt{\sum_{i=1}^d 2q_i \Sigma_{i,i} }\sqrt{  \sum_{i=1}^d  2(4q_i-3) \Big( \E[ F_i^4] - 3 \E[ F_i^2]^2 \Big)   } \notag \\
  &\leq  \sqrt{ 4q_d(4q_d-3) {\rm Tr}(\Sigma)  }\sqrt{  \sum_{i=1}^d   \Big( \E[ F_i^4] - 3 \E[ F_i^2]^2 \Big)   }  \notag \\
  & \leq 4q_d\sqrt{  {\rm Tr}(\Sigma)  } \sum_{i=1}^d \sqrt{    \E[ F_i^4] - 3 \E[ F_i^2]^2  } \, , \label{discrete}
  \end{align}
where the last inequality follows from the elementary fact that $\sqrt{a_1 + \ldots  +a_d}\leq \sqrt{a_1} + \ldots + \sqrt{a_d}$ for any nonnegative reals $a_1,\ldots, a_d$.

\bigskip

Now we consider $\E\big[ \| S \| _{\rm H.S.} \big]$:
\begin{align}
\E\big[ \| S \| _{\rm H.S.} \big] = \E\left( \sqrt{    \sum_{i,j=1}^d S_{i,j}^2  } \,\, \right) & \leq  \left( \sum_{i,j=1}^d \E[ S_{i,j}^2 ] \right)^{1/2} \notag \\
& =  2 \left( \sum_{i,j=1}^d {\rm Var}\big( \Gamma(F_i,F_j) \big) \right)^{1/2}\, .  \label{Metz1}
\end{align}
 It follows from \eqref{DP-0} that 
\begin{align}
 \sum_{i,j=1}^d {\rm Var}\big( \Gamma(F_i,F_j) \big) &\leq  \sum_{i,j=1}^d  \frac{(q_i+q_j-1)^2}{4} \Big(  \E[ F_i^2F_j^2 ] -  2\E[ F_iF_j]^2  -  \Var(F_i) \Var(F_j)    \Big) \notag    \\
 & \leq   \frac{(2q_d -1)^2}{4}   \sum_{i,j=1}^d  \Big(  \E[ F_i^2F_j^2 ] -  2\E[ F_iF_j]^2  -  \Var(F_i) \Var(F_j)    \Big) \notag \\
 & =   \frac{(2q_d -1)^2}{4}  \E\big[ \Enorm{F}^4 - \Enorm{N}^4 \big] \,\, ,  \label{Metz2}
  \end{align}
where the last equality is  a consequence of the fact that (see {\it e.g.} (4.2) in \cite{NR14})
$$\E\big[\Enorm{N}^4 \big] = \sum_{i,j=1}^d \big( \Sigma_{i,i} \Sigma_{j,j} + 2 \Sigma_{i,j}^2 \big) \,\, .$$
 
 \bigskip
 
\begin{lemma}\label{eq1.5}  Let $F, N$ be given as before, then 
\begin{align*}
&\qquad \E\big[ \Enorm{F}^4 \big] - \E \big[ \Enorm{N}^4 \big]  \\
& \leq  2\left(\, \sum_{i=1}^d \sqrt{   \E[F_i^4]-3\, \E[F_i^2]^2 } \, \right)^2 + 2 \left( \sum_{i=1}^{d-1} \sqrt{\E[ F_i^4]} \right) \sum_{j=2}^d  \sqrt{   \E[F_j^4]-3\, \E[F_j^2]^2 }\,\, .
\end{align*}
In particular, if $q_1 = \ldots = q_d$, one has,
\[
\E\big[ \Enorm{F}^4 \big] - \E \big[ \Enorm{N}^4 \big]   \leq  2\left(\, \sum_{i=1}^d \sqrt{   \E[F_i^4]-3\, \E[F_i^2]^2 } \, \right)^2  \, .
\]

\end{lemma}

\begin{proof}  Let us first consider the particular case where $q_1 = \ldots = q_d$.  One obtains from  Lemma \ref{leint2}   that 
\begin{align*}
\E\big[ \Enorm{F}^4 \big] - \E \big[ \Enorm{N}^4 \big]  & = \sum_{i,j=1}^d \Big( \E[ F_i^2F_j^2 ] -  2\E[ F_iF_j]^2  -  \Var(F_i) \Var(F_j)  \Big) \\
 & \leq  2   \sum_{i,j=1}^d \sqrt{  \big( \E[ F_i^4]  - 3\E[ F_i^2]^2 \big)\big( \E[ F_j^4]  - 3\E[ F_j^2]^2 \big)   }     \\
 & = 2\left(\, \sum_{i=1}^d \sqrt{   \E[F_i^4]-3\, \E[F_i^2]^2 } \, \right)^2 \, .
  \end{align*}
 
 In the general case where $q_1\leq \ldots \leq q_d$, Lemma \ref{leint2}  implies   
\begin{align} 
 &\qquad \E\big[ \Enorm{F}^4 \big] - \E \big[ \Enorm{N}^4 \big] \notag   \\
 & =  \sum_{i,j=1}^d \mathbf{1}_{(q_i = q_j)} \Big( {\rm Cov}( F_i^2, F_j^2 ) -  2\E[ F_iF_j]^2   \Big)    +  2 \sum_{1\leq i < j \leq d} \mathbf{1}_{(q_i < q_j)} {\rm Cov}( F_i^2, F_j^2 ) \label{indicator}  \\
 &  \leq  2\left(\, \sum_{i=1}^d \sqrt{   \E[F_i^4]-3\, \E[F_i^2]^2 } \, \right)^2   +   2\sum_{1\leq i <j\leq d} \sqrt{\E[ F_i^4]}\sqrt{\big( \E[ F_j^4] - 3 \E[F_j^2]^2\big)}  \,\, .  \notag
   \end{align}
One can  rewrite ${\displaystyle \sum_{1\leq i < j \leq d} }$ as ${\displaystyle\sum_{j=2}^d \sum_{i=1}^{j-1}}$  and then the desired result follows.  \qedhere

  \end{proof}
  
\begin{remark}
Note that  in the  same way, we can provide another proof of the quantitative Peccati-Tudor Theorem in the Gaussian setting. In particular, keeping the indicator functions in \eqref{indicator}, we can obtain the  bound in a similar form   as in \cite[Theorem 1.5]{NN11}. 
\end{remark}

\bigskip

{\noindent\bf\small End of the proof of Theorem \ref{mtmd}.}  First we obtain from \eqref{Metz1}-\eqref{Metz2} and Lemma \ref{eq1.5} that
\begin{align}
 \E\big[ \| S \| _{\rm H.S.} \big]  & \leq    (2q_d -1) \sqrt{\E\big[ \Enorm{F}^4 - \Enorm{N}^4 \big] } \, \notag\\
 & \leq     \sqrt{2}  (2q_d -1) \sum_{i=1}^d \sqrt{   \E[F_i^4]-3\, \E[F_i^2]^2 } \notag  \\
 & \qquad + \sqrt{2}  (2q_d -1)  \left( \sum_{i=1}^{d-1}  \E[   F_i^4  ]^{1/4}     \right) \sum_{j=2}^d  \big(   \E[F_j^4]-3\, \E[F_j^2]^2 \big)^{1/4}  \, .     \label{notneatanymore}
\end{align}
If $g\in C^3(\R^d)$ and  $g(F), g(N)$ are integrable, then by Proposition \ref{Meckes09},  we deduce 
\begin{align*}
 &\quad \Big\vert  \E  \big[ g(F)  - g(N)  \big]  \Big\vert \\
 & \leq  \Theta_1(g)\, \E[ \| S \| _{\rm H.S.} ] +\Theta_2(g)\, \sqrt{\sum_{i=1}^d 2\Lambda_{i,i} \Sigma_{i,i} + \E[S_{i,i} ] }\sqrt{\sum_{i=1}^d\rho_i(F)} \\
 & \leq       \sqrt{2}  (2q_d -1) \Theta_1(g) \sum_{i=1}^d \sqrt{   \E[F_i^4]-3\, \E[F_i^2]^2 }    \\
 & \qquad +  4q_d \Theta_2(g)\sqrt{  {\rm Tr}(\Sigma)  } \sum_{i=1}^d \sqrt{    \E[ F_i^4] - 3 \E[ F_i^2]^2  } \\
  &\qquad + \sqrt{2}  (2q_d -1) \Theta_1(g)  \left( \sum_{i=1}^{d-1}  \E[   F_i^4  ]^{1/4}     \right) \sum_{j=2}^d  \big(   \E[F_j^4]-3\, \E[F_j^2]^2 \big)^{1/4} \, ,
 \end{align*}
where the last inequality follows from \eqref{notneatanymore} and \eqref{discrete}.      It is easy to check that 
$$
 \sqrt{2} \Theta_1(g)(2q_d-1) + 4q_d \sqrt{ \text{Tr}(\Sigma)} \Theta_2(g)  = B_3(g) \quad\text{and}\quad  \sqrt{2}  (2q_d -1) \Theta_1(g) = A_2(g) \, .
$$
Assertion (i) of Theorem \ref{mtmd} follows immediately.    Assertion (ii) can be proved in the same way,  by using moreover the relations:
$$
 \sqrt{2} K_1(g)(2q_d-1) + 4q_d \sqrt{ \text{Tr}(\Sigma)}  K_2(g)  = B_2(g) \quad\text{and}\quad  \sqrt{2}  (2q_d -1)  K_1(g) = A_1(g) \, .
$$

\bigskip

\begin{remark}\label{rem43} With the  notation and assumptions given as in Theorem \ref{mtmd}, if in addition $q_1 = q_d$, that is,  all the components of the random vector $F$ belong to the same Poisson Wiener chaos, then we can obtain  better bounds, namely:
\begin{enumerate}[{\normalfont (i)}]
 \item For every $g\in C^3(\R^d)$,  we have
 \begin{align*}
  \babs{\E[g(F)]-\E[g(N)]}   \leq     B_3(g) \,\, \sum_{i=1}^d\sqrt{  \E[F_i^4]-3\, \E[F_i^2]^2  } \, .
    \end{align*}

\item If,  in addition, $\Sigma$ is positive definite, then for every $g\in C^2(\R^d)$,   we have
\begin{align*}
  \babs{\E[g(F)]-\E[g(N)]} \leq  B_2(g)\,\, \sum_{i=1}^d\sqrt{  \E[F_i^4]-3\, \E[F_i^2]^2   }  \, .
  \end{align*} 

\end{enumerate}

\end{remark}

\section{Proofs of technical and auxiliary results}\label{prooftech}

In this section, we first  provide the  proofs of Lemma \ref{leint1}, Lemma \ref{leint2}.  The following result from \cite{NR14} will be helpful.

\begin{lemma}[Lemma 2.2 of \cite{NR14}] \label{NR-2.2}
 Given $p,q\in\N$,  $f\in L_s^2(\mu^p)$ and $g\in L^2_s(\mu^q)$,  then 
 \begin{equation*}
(p+q)! \,  \Enorm{f\widetilde{\otimes}g}^2 = p! q! \, \sum_{r=0}^{p\wedge q}\binom{p}{r}\binom{q}{r}\Enorm{f\otimes_r g}^2 \geq  p!q!\, \| f \| _2^2  \| g \| _2^2 + \delta_{p,q} p! q!\, \langle f, g\rangle_2^2\, ,
 \end{equation*}
and in the case of $p=q$, one has
$$ 
(2p)! \big\langle f\widetilde{\otimes} f,   g\widetilde{\otimes} g \big\rangle_2 = 2p!^2\, \langle f, g\rangle_2^2  + \sum_{r=1}^{p-1} p!^2 \binom{p}{r}^2 \big\langle f\otimes_r g,   g\otimes_r f \big\rangle_2\,\, .  
$$
Here we follow the convention that ${\displaystyle \sum_{r=1}^0 = 0}.$
\end{lemma}

\subsection{Proof of Lemma \ref{leint1}} Without loss of  generality, we assume $F = I_p^\eta(f)$ and $G = I_q^\eta(g)$ for some $f\in L^2_s(\mu^p)$ and  $g\in L^2_s(\mu^q)$.  It follows from Lemma \ref{ledp1} and the definition of $\Gamma$ that $J_{p+q}(FG) = I^\eta_{p+q}(f\widetilde{\otimes} g)$ and
\begin{align}
2\, \Gamma(F,G) =  (p+q) \E\big[ FG \big] +  \sum_{k=1}^{p+q-1} (p+q-k)\, J_k(FG)  \, . \label{qq}
\end{align}
By orthogonality, 
\begin{align*}
 {\rm Var}\big( \Gamma(F,G) \big) & =   \frac{1}{4} \sum_{k=1}^{p+q-1} (p+q-k)^2\,  {\rm Var}\big( J_k(FG)   \big) \\
  &\leq    \frac{(p+q-1)^2}{4} \sum_{k=1}^{p+q-1}    {\rm Var}\big( J_k(FG)   \big) \,\,. 
\end{align*}
Similarly, as $FG\in L^2(\Prob)$, we have $FG = \E[ FG] + \sum_{k=1}^{p+q} J_k(FG)$ so that
\begin{align*}
 \E[ F^2G^2 ] & = \E[ FG]^2 + \sum_{k=1}^{p+q - 1} {\rm Var}\big( J_k(FG) \big)  +     {\rm Var}\big( J_{p+q}(FG) \big) \,\, \\
 & = \E[ FG]^2 + \sum_{k=1}^{p+q - 1} {\rm Var}\big( J_k(FG) \big)  +   (p+q)! \| f\widetilde{\otimes} g \| ^2_2 \,\, .
\end{align*}
It follows from Lemma \ref{NR-2.2} that
$$(p+q)! \| f\widetilde{\otimes} g \| ^2_2 \geq p!q! \| f \| ^2_2 \| g \| ^2_2 + \delta_{p,q} p!q! \langle f, g\rangle^2_2 = \E(F^2) \E(G^2) + \E[ FG ]^2 \, .$$
Hence 
\begin{align}
 {\rm Var}\big( \Gamma(F,G) \big) & \leq    \frac{(p+q-1)^2}{4} \sum_{k=1}^{p+q-1}    {\rm Var}\big( J_k(FG)   \big) \notag \\
 & =    \frac{(p+q-1)^2}{4}\left(  \E[ F^2G^2 ] -  \E[ FG]^2  -  (p+q)! \| f\widetilde{\otimes} g \| ^2_2 \right) \notag \\
 & \leq  \frac{(p+q-1)^2}{4}\left(  \E[ F^2G^2 ] -  2\, \E[ FG]^2  -  \E(F^2) \E(G^2)    \right)   \,\, . \label{TWO}
\end{align}
In particular, Lemma \ref{NR-2.2}, applied to $p=q$ and $f=g$, gives us 
\[ (2p)!\| f\widetilde{\otimes} f \|^2_2 =  2 p!^2 \| f \|_2^4   + p!^2 \sum_{r=1}^{p-1} {p\choose r}^2 \| f\otimes_r f \|_2^2  \,\,,  \]
therefore implying 
\begin{align*}
 {\rm Var}\big( \Gamma(F,F) \big) & \leq     \frac{(2p-1)^2}{4}\left(  \E[ F^4 ] -   \E[ F^2]^2  - (2p)!\| f\widetilde{\otimes} f \| ^2_2   \right)   \\
 & =     \frac{(2p-1)^2}{4}\left(  \E[ F^4 ] -  3\, \E[ F^2]^2  - p!^2 \sum_{r=1}^{p-1} {p\choose r}^2 \| f\otimes_r f \| _2^2    \right) \,\, .
\end{align*}
This proves \eqref{DP3.1} and 
\begin{align}\label{zero}  
p!^2 \sum_{r=1}^{p-1} {p\choose r}^2 \| f\otimes_r f \| _2^2 \leq   \E[ F^4 ] -  3\, \E[ F^2]^2 \,\, .  
\end{align}
It is also clear from \eqref{TWO} that 
\begin{align}\label{ONE}
 \sum_{k=1}^{2p-1}  {\rm Var}\big( J_k(F^2) \big) \leq   \E[ F^4 ] -  3\, \E[ F^2]^2 \, .
 \end{align}

It remains to show \eqref{DP3.2} now:  similarly, we write $F^2 =\E[F^2] +  \sum_{k=1}^{2p} J_k(F^2)$ and by \eqref{qq}
\begin{align}\label{SB1}
\Gamma(F,F)  = p\, \E[F^2] + \frac{1}{2} \sum_{k=1}^{2p-1} (2p-k) J_k(F^2) \,\, . 
\end{align}
So by orthogonality, we have  
    \begin{align*}
  - \E[F^4] +  \frac{3}{p}  \E\big[F^2 \, \Gamma(F,F) \big] & = - \E[F^4]  +  3 \E[F^2]^2 + \frac{3}{2p} \sum_{k=1}^{2p-1} (2p-k) \, {\rm Var}\big( J_k(F^2) \big)  \\
  &\leq  - \E[F^4] +   3 \E[F^2]^2 + \frac{3(2p-1)}{2p}  \sum_{k=1}^{2p-1}  {\rm Var}\big( J_k(F^2) \big) \\
  &\leq - \E[F^4] +   3 \E[F^2]^2 + \frac{3(2p-1)}{2p} \big(  \E[F^4] - 3\E[F^2]^2 \big) \\
  & = \frac{4p-3}{2p} \big(  \E[F^4] - 3\E[F^2]^2 \big) \,\, .
  \end{align*}
The other inequality in \eqref{DP3.2} is a trivial consequence of Proposition \ref{GAMMA}-(c). The proof of Lemma \ref{leint1} is complete.

\begin{remark} 

\begin{enumerate}

\item Let $F\in C_p$ have nonzero variance, then we have that $\E[ F^4] >  3\E[F^2]^2$.  
Indeed,  we can always assume $F\in L^4(\Prob)$. If $p=1$, $F = I_1^\eta(f)$ for some $f\in L^2(\mu)$, then  by product formula (see {\it e.g.} Proposition 6.1 in \cite{Lastsv}), one has $\E\big[ I_1^\eta(f)^4 \big] = 3\, \| f\| _2^4 + \int_\mathcal{Z} f(z)^4\, d\mu >  3\E[F^2]^2$. 
 For $p\geq 2$, $F = I_p^\eta(f)$ for some $f\in L^2_s(\mu^p)$, then according to \eqref{zero}, $\E[ F^4]  =  3\E[F^2]^2$ would imply $\| f\otimes_1 f \| _2 = 0$, which would further imply by standard arguments that   $f = 0$ $\mu$-almost everywhere, which is a contradiction to the fact that $F$ is nonzero.   
 
\item Let   $F\in C_p\cap L^4(\Prob)$, one has $p\big( \E[ F^4] - 3 \E[ F^2]^2\big) \leq 6 \Var\big( \Gamma(F,F) \big)$, which shall be compared with  \eqref{DP3.1}. In fact, it follows first from \eqref{DP3.2} that $\E[ F^4] - 3 \E[ F^2]^2\leq 3 \E\big[ F^2\big( p^{-1} \Gamma(F,F) - \E[F^2] \big)\big]$, and by \eqref{SB1} and orthogonality property, we have 
\begin{align*}
\qquad\quad \E\big[ F^2\big(  \Gamma(F,F) - p \E[F^2] \big)\big] & = \frac{1}{2} \sum_{k=1}^{2p-1} (2p-k) \Var\big( J_k(F^2) \big) \\
& \leq   \frac{1}{2} \sum_{k=1}^{2p-1} (2p-k)^2 \Var\big( J_k(F^2) \big) = 2 \Var\big( \Gamma(F,F) \big) \, ,
\end{align*}
 hence $p\big( \E[ F^4] - 3 \E[ F^2]^2\big) \leq 6 \Var\big( \Gamma(F,F) \big)$.
 
\item Let $F, N$ be given as in Theorem \ref{mtmd}, then from \eqref{Metz2} it follows that $\E\big[ \Enorm{F}^4 \big]\geq \E\big[ \Enorm{N}^4 \big]$. Moreover, if one of the components $F_j$ in $F$ has nonzero variance, it follows from the above two points and again \eqref{Metz2} that $\E\big[ \Enorm{F}^4 \big] >  \E\big[ \Enorm{N}^4 \big]$.

\end{enumerate}

\end{remark}

\subsection{Proof of Lemma \ref{leint2}}  Assume  $F= I_p^\eta(f)$ and $G = I_q^\eta(g)$ are in $L^4(\Prob)$ for some $f\in L^2_s(\mu^p)$, $g\in L^2_s(\mu^q)$. Then  it follows from Lemma \ref{ledp1} that  $J_{2p}(F^2) = I_{2p}^\eta(f\widetilde{\otimes} f)$ and  $J_{2q}(G^2) = I_{2q}^\eta(g\widetilde{\otimes} g)$. Moreover, one has
 \begin{align}
 \E[ F^2 G^2 ]  & =  \E\left[ F^2 \sum_{k=0}^{2q} J_k(G^2) \right] \notag  \\
 &= \E\big[ F^2 J_0(G^2) \big] +  \E\big[ F^2 J_{2q}(G^2) \big]   +  \E\left[ F^2 \sum_{k=1}^{2q-1} J_k(G^2) \right]  \notag \\
 & =  \E(F^2) \E(G^2) + \E\big[ F^2\,   J_{2q}(G^2)    \big]  +   \E\left[ F^2 \sum_{k=1}^{2q-1} J_k(G^2) \right] \, . \notag
 \end{align}
If $p < q$, then $ \E\big[ F^2\,   J_{2q}(G^2)    \big]  = 0$, so that 
\begin{align*}
{\rm Cov}(F^2, G^2) & =  \E\left[ F^2 \sum_{k=1}^{2q-1} J_k(G^2) \right]   \leq  \sqrt{\E[F^4]} \sqrt{ \sum_{k=1}^{2q-1} {\rm Var} \big( J_k(G^2) \big)   } \,\,,
\end{align*}
where the above inequality follows from Cauchy-Schwarz inequality and isometry property.  The desired result \eqref{bel1} follows   from \eqref{ONE}.

Now we consider the case where $p=q$, 
 \begin{align*}
 \E\left[ F^2 \sum_{k=1}^{2p-1} J_k(G^2) \right]     & =   \sum_{k=1}^{2p-1}\E\big[ J_k(F^2) J_k(G^2) \big]      \\
   & \leq \sqrt{ \sum_{k=1}^{2p-1}  \Var\big( J_k(F^2)\big)     }  \,\,  \sqrt{ \sum_{k=1}^{2p-1}  \Var\big( J_k(G^2)\big)     }  \,\, \text{\small (by Cauchy-Schwarz)} \\
   & \leq  \sqrt{  \big(  \E[ F^4] - 3\E[F^2]^2 \big) \big(  \E[ G^4] - 3\E[G^2]^2 \big) }    \quad\text{ due to   \eqref{ONE}. }
   \end{align*}
By orthogonality  property, we have 
 \begin{align*}
 \E\big[  J_{2p}(F^2) J_{2p}(G^2)\big] & = (2p)! \big\langle f\widetilde{\otimes} f ,  g\widetilde{\otimes} g \big\rangle_2 \\
 & = 2 p!^2 \langle f, g \rangle^2_2 +  \sum_{r=1}^{p-1} p!^2 {p\choose r}^2 \big\langle f\otimes_r g,  g\otimes_r f \big\rangle_2 \, ,
 \end{align*}
 where the last equality follows from  Lemma \ref{NR-2.2}.  
 
 As a consequence, one has 
 \begin{align*}
 \E\big[ F^2 J_{2p}(G^2)    \big]  - 2 \E[ FG]^2  &=   \sum_{r=1}^{p-1} p!^2 {p\choose r}^2 \big\langle f\otimes_r g,  g\otimes_r f \big\rangle_2  \leq \sum_{r=1}^{p-1} p!^2 {p\choose r}^2 \big\| f\otimes_r g\big\| ^2_2  
  \end{align*}
by Cauchy-Schwarz.  Note that, by definition of contractions and Fubini theorem, we have $\big\| f\otimes_r g\big\| ^2_2 = \big\langle f\otimes_{p-r} f,  g\otimes_{p-r} g \big\rangle_2$ for each $r=1, \ldots, p-1$. Thus, 
 
  \begin{align*}
  &\quad \sum_{r=1}^{p-1} p!^2 {p\choose r}^2 \big\| f\otimes_r g\big\| ^2_2\\
& =  \sum_{r=1}^{p-1} p!^2 {p\choose r}^2 \big\langle f\otimes_{p-r} f,  g\otimes_{p-r} g \big\rangle_2 = \sum_{r=1}^{p-1} p!^2 {p\choose r}^2 \big\langle f\otimes_{r} f,  g\otimes_{r} g \big\rangle_2 \\
  &\leq   \sum_{r=1}^{p-1} p!^2 {p\choose r}^2 \| f\otimes_{r} f\| _2    \|  g\otimes_{r} g \| _2 \quad\text{(by Cauchy-Schwarz)}  \\
  &\leq  \sqrt{   \sum_{r=1}^{p-1} p!^2 {p\choose r}^2 \| f\otimes_{r} f\| ^2_2 }  \sqrt{   \sum_{r=1}^{p-1} p!^2 {p\choose r}^2 \| g\otimes_{r} g\| ^2_2 }  \quad\text{(by Cauchy-Schwarz)} \\
     &\leq \sqrt{\E[ F^4] - 3 \, \E[F^2]^2} \sqrt{\E[ G^4] - 3 \, \E[G^2]^2} \quad\text{due to \eqref{zero}.}
 \end{align*}
 Hence, we obtain
 \begin{align*}
 {\rm Cov}(F^2, G^2)  -  2\, \E[ FG]^2 & =  \E\big[ F^2\,   J_{2p}(G^2)    \big]  - 2\, \E[ FG]^2  +   \E\left[ F^2 \sum_{k=1}^{2p-1} J_k(G^2) \right]  \\
 &\leq  2 \sqrt{\E[ F^4] - 3 \, \E[F^2]^2} \sqrt{\E[ G^4] - 3 \, \E[G^2]^2} \,\, .
 \end{align*}
 The proof is completed.

\subsection{Proof of Proposition \ref{trans-p}}  It follows from \eqref{zero} that 
 \[
 p!^2 \sum_{r=1}^{p-1} {p\choose r}^2 \| f_n\otimes_r f_n \| _2^2 \leq   \E[ I_p^\eta(f_n)^4 ] -  3\, \E[ I_p^\eta(f_n)^2]^2 \,\, .  
 \]
If $ \E[ I_p^\eta(f_n)^4 ] \to 3$ as $n\to+\infty$, then $ \| f_n\otimes_r f_n \| _2\to 0$ for each $r \in\{ 1, \ldots, p-1\}$. Therefore by   Theorem \ref{FMT-NP}, $ \E[ I_p^W(f_n)^4 ] \to 3$ and moreover by \eqref{npbound},
$$\lim_{n\to+\infty} d_{\rm TV}\big( I_p^W(f_n), N \big)  = 0 \,\, . $$
This completes the proof of our transfer principle.

\subsection{Proof of Theorem \ref{univ-thm}}

The equivalence of $(A_1)$ and $(A_2)$ is the content of Theorem 7.5 in \cite{NPR-aop}.  For each $i\in\N$, define 
\[
g_i= \frac{1}{  \sqrt{t_{i+1} - t_i    } } \mathbf{1}_{[t_i, t_{i+1} ) } ,
\]
then the homogeneous sum $Q_d(f, N, \mathbf{P})$, defined according to \eqref{sums}, can be expressed as the $d$-th multiple integral $I^\eta_d(\widehat{f} )$, where 
\begin{align}\label{f-hat}
\widehat{f} : = \sum_{1\leq i_1, \ldots, i_d\leq N} f(i_1, \ldots, i_d) g_{i_1}\otimes \cdots \otimes g_{i_d} \,\, .
\end{align}
From now on, we identify $\widehat{f}$ with $f$ in case of no confusion.     Observe that the sequence $\mathbf{G}$ of {\it i.i.d} standard Gaussian random variables can be realised via the Brownian motion $(W_t, t\in\R_+)$. That is, for each $i\in\N$, we put $G_i =  I_1^W(g_i)$. As a consequence,  the homogeneous sum $Q_d(f, N, \mathbf{G})$ can be rewritten as $I_d^W(\widehat{f} )$, with $\widehat{f}$ given in \eqref{f-hat}.

With these notions at hand and  in view of our transfer principle, if $(A_0)$ holds, then, for each $j\in\{1, \ldots, d\}$ and every $r \in\{ 1, \ldots, q_j-1 \}$, 
$\big\| f_{n,j}\otimes_r  f_{n,j} \big\| _2 \to 0
$, as $n\to+\infty$.    Then, $(A_1)$ is an immediate consequence of   Theorem \ref{PTudor-G}.

Finally, it is known that the fourth {\it central} moment of a Poisson random variable with parameter $\lambda\in(0,+\infty)$ is given by 
$
\lambda(1+ 3\lambda)$, then $\E[ P_i^4 ] = 3 + (t_{i+1} - t_i  )^{-1}
$. 
If $\inf\{ t_{i+1} - t_i \,:\, i\in\N \} > 0$, then Jensen's inequality implies
\[
\sup_{i\in\N} \E\big[ \vert P_i \vert^3 \big] \leq  \sup_{i\in\N} \E\big[ \vert P_i \vert^4 \big]^{3/4}  < +\infty\, .
\]
 Hence, we obtain the implication ``$(A_2)\Rightarrow (A_3)$'', while   the implication ``$(A_3)\Rightarrow (A_0)$'' is a consequence of Theorem 3.4 in \cite{PZ2}. The proof of Theorem \ref{univ-thm} is finished.
 
\bigskip

\subsection{Proof of Proposition \ref{Meckes06}} Without loss of any generality, we may and will assume that $\text{Var} (Y) = 1$ and $N\sim\mathcal{N}(0,1)$. Let $f:\R\to\R$ be $1$-Lipschitz function and  consider 
$$g(x) = e^{x^2/2} \int_{-\infty}^x  \big( f(y) - \E[f(N)] \big) e^{-y^2/2} \, dy \,\, , \quad x\in\R\, , $$
which satisfies the  {\it Stein's equation}
\begin{equation}\label{steineq}
g'(x) - xg(x) = f(x) - \E[ f(N) ]
\end{equation}
as well as $\| g' \| _\infty \leq  \sqrt{2/\pi} $, $\| g''\| _\infty \leq 2$, see {\it e.g.} Section 2.3 in \cite{Zheng15}.
In what follows, we fix such a pair $(f, g)$ of functions. Let $G: \R\to \R$ be a differentiable function such that $G' = g$. Then  due to $Y_t\overset{law}{=} Y$ and $Y\in L^4(\Prob)$, one has 
\begin{align*}
0 = \E\big[ G(Y_t) - G(Y) \big] = \E\big[ g(Y) (Y_t - Y) + \frac{1}{2} g'(Y) (Y_t - Y)^2 \big] + \E[ R_g ]
\end{align*}
with $ |R_g | \leq \dfrac{1}{6} \| g'' \| _\infty\,   \vert Y_t - Y \vert^3 $.  
 It follows that  
$$  0  =  \E\left[ g(Y) \times \frac{1}{t}\, \E\big[ Y_t - Y | \mathscr{G} \big] \right] +\frac{1}{2}\,  \E\left[ g'(Y)  \times \frac{1}{t}\, \E\big[ (Y_t - Y)^2 | \mathscr{G} \big] \right]  + \frac{1}{t}\, \E[ R_g  ]. $$ 
 By assumption (c) and as $t\downarrow 0$,
\begin{align*}
\left\vert \frac1t \, E [ R_g   ]\right\vert  \leq   \frac{1}{3t} \, E\big[ \vert Y_t - Y \vert^3  \big] & \leq \frac{1}{3} \sqrt{\frac{1}{t} \E\big[   ( Y_t - Y )^2 \big]  } \sqrt{\frac{1}{t} \E\big[   ( Y_t - Y )^4 \big]  } \\
&\to   \frac{1}{3}  \sqrt{   2\lambda+\E[S]} \sqrt{\rho(Y)}
\end{align*}
Therefore as $t\downarrow 0$, assumptions (a) and (b) imply that 
$$ 0 =   \lambda \E\big[ g'(Y) - Yg(Y) \big] + \frac{1}{2}\, \E\big[ g'(Y)S \big] +  \lim_{t\downarrow 0} \frac{1}{t} \E[ R_g ] \,\, .  $$
The above equation shall be understood as ``the limit $\lim_{t\downarrow 0} t^{-1} \E[ R_g ] $ exists and is equal to $- \lambda \E\big[ g'(Y) - Yg(Y) \big] - \frac{1}{2}\, \E\big[ g'(Y)S \big]$, bounded by $\frac{1}{3}  \sqrt{   2\lambda+\E[S]} \sqrt{\rho(Y)}$.

Plugging this into the Stein's equation (\ref{steineq}), we deduce the desired conclusion, namely 
 \begin{align*}
  d_\W(Y, N) & = \sup_{f\in \Lip(1)} \Big\vert \E[ f(Y)- f(N) ] \Big\vert    \leq    \sup_{\substack{ \|g'\|_\infty\leq\sqrt{2/\pi}  \\ \| g'' \| _\infty \leq 2 } } \Big\vert \E[ g'(Y)-  Yg(Y) ] \Big\vert     \\
  &  \leq  \sup_{\substack{ \|g'\|_\infty\leq\sqrt{2/\pi}  \\ \| g'' \| _\infty \leq 2 } } \left ( \,\, \frac{\| g' \| _\infty}{2\lambda} \E\big[ | S | \big] + \left|  \frac{1}{\lambda}  \lim_{t\downarrow 0}\frac{1}{t} \E[ R_g ]    \right|   \right) \\
  &\leq \frac{1}{\lambda\sqrt{2\pi}} \E\big[|S|\big] + \frac{ \sqrt{   2\lambda+\E[S]}}{3\lambda}   \sqrt{\rho(Y)} \,\, . 
  \end{align*}
 The general case follows from the fact that $d_\W(Y, N) = \sigma\, d_\W(Y/\sigma, N/\sigma)$ for $\sigma > 0$.

\subsection{Proof of Proposition \ref{Meckes09}}  By the same argument as in the proof of Theorem 3 in \cite{Meck09}, we can assume $g\in C^\infty(\R^d)$ and define 
$$
f(x)=\int_0^1 \frac{1}{2t} \Big(  \E\big[   g(\sqrt{t}\,x+\sqrt{1-t}\,N) \big]-\E[g(N)]  \Big)\,dt \,\, ,
$$
which is a solution to the following  {\it Stein's equation}
\begin{equation}\label{multiStein}
\langle  x,\nabla f(x)\rangle - \langle   \text{Hess} f(x),\Sigma\rangle_\text{H.S.} = g(x)-\E [g(N)] \,\, .
\end{equation}
It is known that  $M_r(f)\leq r^{-1}\, M_r(g)$ for $r=1,2,3$ and $\widetilde{M}_2(f)\leq\frac{1}{2}\widetilde{M}_2(g)$. In particular, if $\Sigma$ is positive definite, then $\widetilde{M}_2(f)\leq\sqrt{2/\pi} \, \Opnorm{\Sigma^{-1/2}} \,M_1(g)$ and $M_3(f)\leq \sqrt{2\,\pi}\, \Opnorm{\Sigma^{-1/2}}\,M_2(g)/4$, see \cite[Lemma 2]{Meck09}.

\bigskip
 
Again, it follows from the same arguments as in  \cite{Meck09} that
\begin{flalign}
0 &=\frac{1}{t}\,\E\left[ \,\,  \frac{1}{2} \Big\langle \text{Hess} f(X),\Lambda^{-1}(X_t-X)(X_t-X)^T  \Big\rangle_\text{H.S.}  \right]  \notag \\
   &\qquad\qquad\qquad + \frac{1}{t} \E\Big[\, \big\langle \Lambda^{-1}(X_t-X),\nabla f(X)\big\rangle  \Big] +  \frac{1}{2t} \E[ R ] ,\label{startMecM}
\end{flalign} 
where $R$ is the error in the Taylor approximation satisfying
\begin{flalign*}
\abs{R}    
& \leq  \frac{1}{3}\Opnorm{\Lambda^{-1}}  \Enorm{X_t-X}^3  \beta  \leq   \frac{\sqrt{d}}{3}\Opnorm{\Lambda^{-1}}  \beta   \sqrt{\sum_{i=1}^d (X_{i,t} - X_i)^2 } \sqrt{  \sum_{i=1}^d (X_{i,t} - X_i)^4 }\,\, ,
\end{flalign*} 
where $ \beta : =  \min\big\{  M_3(g)/3,  \sqrt{2\pi} \Opnorm{\Sigma^{-1/2}}  \,M_2(g)/4\big\}$, 
and the last inequality follows from the elementary inequality $\Enorm{ x - y }^2 \leq \sqrt{d}  \big(  \sum_{i=1}^d (x_{i} - y_i)^4 \big)^{1/2}$ for $x,y\in\R^d$. 

Notice meanwhile that  the assumptions (a) and (b) imply that the limit $t^{-1} \E[ R ]$, as $t\downarrow 0$, is well defined and 
\begin{align*}
- \lim_{t\downarrow 0} \frac{1}{2t} \E[ R] &= \E\Big[  \big\langle  \text{Hess} f(X), \Sigma \big\rangle_\text{H.S.} - \big\langle X, \nabla f(X) \big\rangle \Big] + \frac{1}{2} \E\Big[  \big\langle \text{Hess} f(X), \Lambda^{-1} S \big\rangle_\text{H.S.} \Big] \\
& = \E\big[g(N) - g(X)\big] + \frac{1}{2} \E\Big[  \big\langle \text{Hess} f(X), \Lambda^{-1} S \big\rangle_\text{H.S.} \Big] \,\,,
\end{align*}
where the last equality comes from the definition of Stein's equation. Moreover, by  assumption (c) and the above inequality, we have 
\begin{flalign*}
  \left\vert\lim_{t\to0}\frac{1}{t}\,\E[R ] \right\vert &  \leq  \frac{\sqrt{d}}{3}\Opnorm{\Lambda^{-1}}   \, \beta \,   \sqrt{ \lim_{t\downarrow 0}\frac{1}{t} \E \sum_{i=1}^d (X_{i,t} - X_i)^2 } \sqrt{ \lim_{t\downarrow 0}\frac{1}{t} \E \sum_{i=1}^d (X_{i,t} - X_i)^4 } \\
& =  \frac{\sqrt{d}}{3}\Opnorm{\Lambda^{-1}}  \, \beta \,  \sqrt{\sum_{i=1}^d 2\Lambda_{i,i} \Sigma_{i,i} + \E[S_{i,i} ] }\sqrt{\sum_{i=1}^d\rho_i(X)},
\end{flalign*} 
 where the last equality follows from assumptions (b) and (c). To conclude our proof,  it suffices to notice that $ \E\big[  \langle \text{Hess} f(X), \Lambda^{-1} S \big\rangle_\text{H.S.} \big]$ is bounded by 
 $$ \min\left\{    \frac{1}{2}\widetilde{M}_2(g),\sqrt{\frac{2}{\pi}} \, \Opnorm{\Sigma^{-1/2}}    \,M_1(g)\right\} \Opnorm{\Lambda^{-1}}  \,  \E\big[  \| S \| _\text{H.S.}\big]  \,\, .  $$

\bigskip

\normalem
\bibliography{MultPoisson}
\bibliographystyle{plain}
\end{document}